\documentclass[11pt]{amsart}
\usepackage{amssymb, amsfonts, amsmath, amsthm}
\usepackage{tikz}
 
\newtheorem{theorem}{Theorem}[section]
\newtheorem{lemma}[theorem]{Lemma}
\newtheorem{proposition}[theorem]{Proposition}
\newtheorem{corollary}[theorem]{Corollary}
\newtheorem{prop-and-def}[theorem]{Proposition and Definition}

\theoremstyle{definition}
\newtheorem{definition}[theorem]{Definition}
\newtheorem{notation}[theorem]{Notation}
\newtheorem{remark}[theorem]{Remark}
\newtheorem{defrem}[theorem]{Definition and Remark}
\newtheorem{remnotation}[theorem]{Remark and Notation}

\newtheorem{ad-hoc}[theorem]{ }
\numberwithin{equation}{section}

\newcommand{\cA}{ {\mathcal A} }
\newcommand{\Dalg}{ {\mathcal D}_{\mathrm{alg}} }
\newcommand{\cM}{ {\mathcal M} }
\newcommand{\cP}{ {\mathcal P} }
\newcommand{\cS}{ {\mathcal S} }
\newcommand{\cV}{ {\mathcal V} }

\newcommand{\bC}{ {\mathbb C} }
\newcommand{\bN}{ {\mathbb N} }
\newcommand{\bR}{ {\mathbb R} }
\newcommand{\bZ}{ {\mathbb Z} }

\newcommand{\fM}{ {\mathfrak M} }
\newcommand{\fMcV}{ \fM_{\cV} }

\newcommand{\uBeta}{ \underline{\beta} }
\newcommand{\uKappa}{ \underline{\kappa} }
\newcommand{\uChi}{ \underline{\chi} }
\newcommand{\uPhi}{ \underline{\varphi} }

\newcommand{\uLambda}{ \underline{\lambda} }

\newcommand{\uPsi}{ \underline{\psi} }

\newcommand{\Abs}{ \mbox{Abs} }
\newcommand{\Cat}{ \mathrm{Cat} }
\newcommand{\Int}{ \mbox{Int} }

\newcommand{\Moeb}{ \mbox{M\"ob} }

\newcommand{\Out}{ \mbox{Out} }
\newcommand{\term}{ \mbox{term} }

\newcommand{\piodd}{ \pi^{(\mathrm{odd})} }
\newcommand{\sigmaeven}{ \sigma^{(\mathrm{even})} }

\newcommand{\Deltastar}{ \Delta^{*} }
\newcommand{\NCBopp}{ NC^{(B-opp)} }
\newcommand{\NCB}{ NC^{(B)} }
\newcommand{\boxplusc}{ \boxplus_{c} }
\newcommand{\boxplusB}{ \boxplus_{B} }
\newcommand{\freeprodc}{ \star_{c} }
\newcommand{\freeprodB}{ \star_{B} }
\newcommand{\ecdef}{ \stackrel{def}{\Longleftrightarrow} }

\newcommand{\cut}{\operatorname{cut}}
\newcommand{\attach}{\operatorname{attach}}

\title{A construction which relates c-freeness to infinitesimal freeness}

\author[M. Fevrier]{Maxime Fevrier}
\address{Maxime Fevrier: Laboratoire de Math\'ematiques d'Orsay,
Universit\'e Paris Sud, CNRS, 
Universit\'e Paris-Saclay, 91405 Orsay, France.}
\email{maxime.fevrier@u-psud.fr}
\author[M. Mastnak]{Mitja Mastnak}
\address{Mitja Mastnak: Department of Mathematics \& Computing Science, 
        Saint Mary's University, Halifax, Nova Scotia B3H 3C3, Canada.}
\email{mmastnak@cs.smu.ca}
\author[A. Nica]{Alexandru Nica}
\thanks{MM, AN: research supported by a Discovery Grant from 
	NSERC, Canada.}
\address{Alexandru Nica: Department of Pure Mathematics, 
	University of Waterloo, Ontario, Canada.}
\email{anica@uwaterloo.ca}
\author[K. Szpojankowski]{Kamil Szpojankowski}
\thanks{KSz: research partially suported by NCN grant 2016/23/D/ST1/01077}
\address{Kamil Szpojankowski:
	Faculty of Mathematics and Information Science,
	Warsaw University of Technology, Poland.}
\email{k.szpojankowski@mini.pw.edu.pl}

\begin{document}

\begin{abstract}
We consider two extensions of free probability that have been
studied in the research literature, and are based on the notions of
{\em c-freeness} and respectively of {\em infinitesimal freeness} for
noncommutative random variables.  In a 2012 paper, Belinschi and 
Shlyakhtenko pointed out a connection between these two frameworks, at
the level of their operations of 1-dimensional free additive 
convolution.  Motivated by that, we propose a construction which 
produces a multi-variate version of the Belinschi-Shlyakhtenko result, 
together with a result concerning free products of multi-variate 
noncommutative distributions. 
Our arguments are based on the combinatorics of the specific types
of cumulants used in c-free and in infinitesimal free probability.
They work in a rather general setting, where the initial data 
consists of a vector space $\cV$ given together with a linear map 
$\Delta : \cV \to \cV \otimes \cV$.  In this setting, all the needed 
brands of cumulants live in the guise of families of multilinear 
functionals on $\cV$, and our main result concerns a certain 
transformation $\Deltastar$ on such families of multilinear functionals.
\end{abstract}

\maketitle

\section{Introduction}

In this paper we study the relation between two extensions of the 
framework of free probability which have been considered in the research 
literature, and are based on the notions of {\em c-free independence} 
and respectively of {\em infinitesimal free independence} for
noncommutative random variables.  

$\ $

\subsection{Description of framework.}

$\ $

\noindent
We consider the plain algebraic framework of a 
``noncommutative probability space'', or {\em ncps} for short,
by which we simply understand a pair 
$( \cA , \varphi )$ with $\cA$ a unital algebra over $\bC$ and 
$\varphi : \cA \to \bC$ a linear functional having
$\varphi ( 1_{{ }_{\cA}} ) = 1$.  The fundamental concept of interest 
for us in this framework is the one of {\em free independence} for 
a family of unital subalgebras $\cA_1, \ldots , \cA_k \subseteq \cA$.  
A handy tool for the combinatorial study of free independence
is a sequence of multilinear functionals
$( \kappa_n : \cA^n \to \bC )_{n=1}^{\infty}$, called the 
{\em free cumulant functionals} associated to $( \cA, \varphi )$; 
indeed, the free independence of $\cA_1, \ldots , \cA_k$ can be 
conveniently re-phrased as a vanishing condition of the ``mixed'' 
free cumulants with entries in these algebras.  For a basic 
exposition of how free cumulants are used in the study of free 
independence, see e.g. Lecture 11 of \cite{NiSp2006}.

In this paper we consider two extensions of the free probability 
framework.

One of the extensions is to the framework of {\em c-free} 
(shorthand for {\em ``conditionally free''}) independence, which was 
initiated in \cite{BoSp1991, BoLeSp1996}, and has accumulated 
a fairly large amount of work since then (see e.g. 
\cite{BoBr2009, PoWa2011} and the references indicated there).  
Here one works with triples $( \cA , \varphi , \chi )$
where $( \cA , \varphi )$ is an ncps (as above) and 
$\chi : \cA \to \bC$ is an additional linear functional with
$\chi ( 1_{{ }_{\cA}} ) = 1$.  We will refer to such a triple 
$( \cA , \varphi , \chi )$ by calling it a {\em C-ncps}.  When 
dealing with a C-ncps, the fundamental concept of interest is
the one of {\em c-free independence} (with respect to $\varphi$ 
and $\chi$)  for a family of unital subalgebras 
$\cA_1, \ldots , \cA_k \subseteq \cA$; this amounts to asking 
that $\cA_1, \ldots , \cA_k$ are freely independent in the usual 
sense with respect to $\varphi$, and in addition, that a 
multiplicativity condition on $\chi$ is fulfilled for certain 
special products formed with elements from 
$\cA_1, \ldots , \cA_k$.  The paper \cite{BoLeSp1996} also 
introduced a recipe for how to define a family
$( \kappa^{(c)}_n : \cA^n \to {\mathbb C} )_{n=1}^{\infty}$ 
of {\em c-free cumulant functionals} associated to 
$( \cA , \varphi, \chi )$.  In terms of cumulants, the c-free 
independence of a family of unital subalgebras 
$\cA_1, \ldots , \cA_k \subseteq \cA$ is equivalent to: usual 
free independence of $\cA_1 , \ldots , \cA_k$ with respect to 
$\varphi$, plus a vanishing condition on the mixed 
$\kappa^{(c)}_n$ cumulants.  

The other extension we want to consider is to the framework of 
{\em infinitesimal free independence}.  This was introduced in 
\cite{BeSh2012}, with some earlier combinatorial ideas around
this topic appearing in \cite{BiGoNi2003}.  The literature on 
infinitesimal free probability is not extensive, but this is 
nevertheless a promising direction of research, e.g. due to 
its connections to random matrix theory (\cite{Sh2018}, see also
the discussion in \cite{Mi2018}).
In order to study infinitesimal free independence, one works 
with triples $( \cA , \varphi , \varphi ' )$
where $( \cA , \varphi )$ is an ncps (as above) and 
$\varphi ' : \cA \to \bC$ is an additional linear functional
such that $\varphi ' ( 1_{{ }_{\cA}} ) = 0$.  We will refer to 
such a triple $( \cA , \varphi , \varphi ' )$ by calling it 
an {\em I-ncps}.  In relation to an I-ncps, the fundamental 
concept of interest is the one of 
{\em infinitesimal free independence} (with respect to $\varphi$ 
and $\varphi '$) for a family of unital subalgebras 
$\cA_1, \ldots , \cA_k \subseteq \cA$; this amounts to asking that 
$\cA_1, \ldots , \cA_k$ are freely independent in the usual sense 
with respect to $\varphi$, and in addition, that a derivation-type 
condition on $\varphi '$ is fulfilled for certain 
special products formed with elements from $\cA_1, \ldots , \cA_k$.  
One has, as found in \cite{FeNi2010}, a natural way of defining 
a family $( \kappa_n ' : \cA^n \to \bC )_{n=1}^{\infty}$ of 
{\em infinitesimal free cumulant functionals} associated to 
$( \cA , \varphi, \varphi ' )$, obtained essentially by ``taking a 
formal derivative with respect to $\varphi$'' in the formulas 
describing the free cumulants of $\varphi$.  In terms of cumulants, 
the infinitesimal free independence of a family of unital subalgebras 
$\cA_1, \ldots , \cA_k \subseteq \cA$ is equivalent to: usual 
free independence of $\cA_1 , \ldots , \cA_k$ with respect to 
$\varphi$, plus a vanishing condition on the mixed 
$\kappa_n '$ cumulants.  

$\ $

\subsection{The map $\Psi$ from \cite{BeSh2012}, and its 
multivariate extension.}

$\ $

\noindent
The main objective of the present paper is to look for connections 
between the frameworks of c-free independence and of infinitesimal
free independence.  Our starting point is a result of Belinschi 
and Shlyakhtenko, Theorem 22 in \cite{BeSh2012}, which goes at 
the level of {\em distributions}.  
For our purposes here, distributions will be assumed to 
have finite moments of all orders, and will be viewed as 
linear functionals on the algebra of polynomials $\bC [X]$;  
the C-ncps and I-ncps considered in this context 
are thus of the form 
\begin{equation}   \label{eqn:1a}
( \bC [X], \mu , \nu ) \mbox{ and respectively }
( \bC [X], \mu , \mu ' ), 
\end{equation}
with $\mu, \nu , \mu ' : \bC [X] \to \bC$ linear such that 
$\mu (1) = \nu (1) = 1$, $\mu ' (1) = 0$.  For pairs 
$( \mu , \nu )$ and $( \mu , \mu ' )$ as in (\ref{eqn:1a})
one naturally defines operations of 
{\em free additive convolution}, denoted by $\boxplusc$ 
and 
\footnote{ The second of the two notations comes from
the fact that $\boxplusB$ is also known as 
``free additive convolution of type B''.}
respectively 
$\boxplusB$, which reflect the operation of adding c-free 
(respectively infinitesimally free) 
elements in a general C-ncps (respectively I-ncps).  

Theorem 22 of \cite{BeSh2012} gives a connection between the 
operations of free additive convolution $\boxplusc$ and 
$\boxplusB$.  It is phrased in terms of a certain map
\[
\Psi : \cM \to \cM '
\]
(introduced in the same paper \cite{BeSh2012}), 
where $\cM$ is a space of probability distributions on 
$\bR$ and $\cM '$ is a space of signed measures on $\bR$, 
with $\mu ' ( \bR ) = 0$, $\forall \, \mu ' \in \cM '$.  
For a $\nu \in \cM$ with compact support, the definition 
of $\mu ' := \Psi ( \nu ) \in \cM '$ can be given in the 
guise of a formula describing the moments of $\mu '$, namely:
\begin{equation}   \label{eqn:1b}
\int_{\bR} t^n \ d \mu ' (t) = n \beta_{n+1; \nu}, 
\ \ n \in \bN,
\end{equation}
where $( \beta_{n; \nu} )_{n=1}^{\infty}$ is the 
sequence of {\em Boolean cumulants} of $\nu$ (some 
siblings of free cumulants, which come from the parallel 
world of so-called ``Boolean probability''). 
Referring to the map $\Psi$, Theorem 22 of \cite{BeSh2012} 
can be stated as follows: if we consider probability measures 
$\mu_1, \nu_1, \mu_2, \nu_2 \in \cM$ and we put 
\begin{equation}   \label{eqn:1c}
( \mu_1 , \nu_1 ) \boxplusc 
( \mu_2 , \nu_2 ) = ( \mu , \nu ) ,
\end{equation}
then it follows that
\begin{equation}   \label{eqn:1d}
( \mu_1 , \Psi ( \nu_1 ) ) \boxplusB 
( \mu_2 , \Psi ( \nu_2 ) ) = ( \mu , \Psi ( \nu ) ) 
\in \cM \times \cM ' .
\end{equation}

Clearly, one can consider a plain algebraic version of the 
implication ``(\ref{eqn:1c}) $\Rightarrow$ (\ref{eqn:1d})'',
where $\cM$ and $\cM '$ are replaced with spaces of linear 
functionals on $\bC [X]$.  In the present paper we show that, 
in this algebraic version, the map $\Psi$ and the said 
implication 

\noindent
``(\ref{eqn:1c}) $\Rightarrow$ (\ref{eqn:1d})''
can be generalized to a multivariate framework where instead 
of $\bC [X]$ we use the algebra 
$\bC \langle X_1, \ldots , X_k \rangle$ of polynomials 
in non-commuting indeterminates $X_1, \ldots , X_k$.
More precisely: instead of $\cM$ and $\cM '$ we will use the 
spaces of ``distributions'' (in plain algebraic sense)
$\Dalg (k)$ and $\Dalg ' (k)$ defined by
\begin{equation}   \label{eqn:1e}
\left\{   \begin{array}{lll}
\Dalg (k)  & := & 
      \{ \mu : \bC \langle X_1, \ldots , X_k \rangle
      : \mu \mbox{ is linear, } \mu (1) = 1 \},       \\
           &    &                                     \\
\Dalg '(k) & := & 
      \{ \mu ' : \bC \langle X_1, \ldots , X_k \rangle
      : \mu ' \mbox{ is linear, } \mu ' (1) = 0 \},
\end{array}  \right.
\end{equation}
and we will use a map
\[
\Psi_k : \Dalg (k) \to \Dalg ' (k)
\]
defined as follows.

\begin{definition}   \label{def:11}
Let $k$ be a positive integer and let $\nu$ be in $\Dalg (k)$.
Then $\Psi_k ( \nu )$ is the linear functional 
$\mu ' \in \Dalg ' (k)$ determined by the requirement that
\begin{equation}  \label{eqn:11a}
\mu ' ( X_{i_1} \cdots X_{i_n} ) =
\sum_{m=1}^n \beta_{n+1; \nu} 
( X_{i_m}, \ldots , X_{i_n}, X_{i_1}, \ldots , X_{i_m} ),
\end{equation}
holding for every $n \in \bN$ and 
$i_1, \ldots , i_n \in \{ 1, \ldots , k \}$, and where 
$\beta_{n+1; \nu}$ denotes the $(n+1)$-th Boolean cumulant 
functional of $\nu$ (a multilinear functional on 
$\bC \langle X_1, \ldots , X_k \rangle^{n+1}$ -- the
precise definition of Boolean cumulants is reviewed in 
Section 3 below).
\end{definition}

It is immediate that in the case $k=1$, Equation (\ref{eqn:11a})
boils down to (\ref{eqn:1b}).

For general $k \in \bN$, the link between the $k$-variate 
versions of the operations $\boxplusc$ and $\boxplusB$ is
stated in the same way as in the univariate case, with the detail 
that the first functional ``$\mu$'' is assumed to be tracial 
(that is, it has the property that $\mu (PQ) = \mu (QP)$ for 
all $P,Q \in \bC \langle X_1, \ldots , X_k \rangle$).

\begin{theorem}   \label{thm:12}
Let $k$ be a positive integer and let 
$\mu_1, \nu_1, \mu_2, \nu_2 \in \Dalg (k)$, 
where $\mu_1$ and $\mu_2$ are tracial.  If we denote
\[ 
( \mu_1 , \nu_1 ) \boxplusc 
( \mu_2 , \nu_2 ) = ( \mu , \nu ) 
\in \Dalg (k) \times \Dalg (k),
\]
then it follows that
\[
( \mu_1 , \Psi_k ( \nu_1 ) ) \boxplusB 
( \mu_2 , \Psi_k ( \nu_2 ) ) 
= ( \mu , \Psi_k ( \nu ) ) 
\in \Dalg (k) \times \Dalg ' (k).
\]
\end{theorem}

Closely related to Theorem \ref{thm:12}, one has a result about 
the actual concepts of independence that are under discussion.  
At the level of multi-variate distributions, 
the statement of this result concerns the suitable versions of 
{\em free products} of algebraic distributions -- more precisely, 
one calls on the free product operations ``$\freeprodc$'' and 
``$\freeprodB$'' which correspond to the operations of 
concatenating c-free tuples (and respectively infinitesimally 
free tuples) in a general C-ncps (respectively I-ncps).  
The result about how the maps $\Psi_k$ link these free product 
operations is stated as follows.

\begin{theorem}   \label{thm:13}
Let $k, \ell$ be positive integers and let 
$\mu_1, \nu_1  \in \Dalg (k)$,
$\mu_2, \nu_2 \in \Dalg ( \ell )$, where 
$\mu_1, \mu_2$ are tracial.  Consider the free product
\[
( \mu_1 , \nu_1 ) \freeprodc ( \mu_2 , \nu_2 ) 
= ( \widetilde{\mu} , \widetilde{\nu} )
\in \Dalg ( k + \ell ) \times \Dalg ( k + \ell ).
\]
 Then one has 
\[
( \mu_1 , \Psi_k ( \nu_1 ) ) \freeprodB
( \mu_2 , \Psi_{\ell} ( \nu_2 ) ) = 
( \widetilde{\mu} , \Psi_{k + \ell} ( \widetilde{\nu} ) ) 
\in \Dalg ( k + \ell ) \times \Dalg ' (k + \ell ).
\]
\end{theorem}

Theorems \ref{thm:12} and \ref{thm:13} are, in turn,
consequences of the next result about how $c$-free cumulants and 
infinitesimal free cumulants get to be related in connection 
to $\Psi_k$.

\begin{theorem}  \label{thm:14}
Let $k$ be a positive integer, let $\mu, \nu$ be in $\Dalg (k)$ 
such that $\mu$ is tracial, and let 
$\mu' := \Psi_k ( \nu ) \in \Dalg ' (k)$.  Then the 
infinitesimal free cumulants $( \kappa_n ' )_{n=1}^{\infty}$
of $( \mu , \mu' )$ are related to the $c$-free cumulants 
$( \kappa_n^{(c)} )_{n=1}^{\infty}$ of $( \mu , \nu )$
by the formula 
\[
\kappa_n ' ( X_{i_1}, \ldots , X_{i_n} ) 
= \sum_{m=1}^n
\kappa_{n+1}^{(c)} 
( X_{i_m}, \ldots , X_{i_n}, X_{i_1}, \ldots , X_{i_m} ),
\] 
holding for every $n \in \bN$ and 
$i_1, \ldots , i_n \in \{ 1, \ldots , k \}$.
\end{theorem}

$\ $

\subsection{A further generalization of the map $\Psi_k$ and 
of Theorem \ref{thm:14}.}

$\ $

\noindent
A careful look at the proof of Theorem \ref{thm:14} 
reveals that the arguments needed there can be presented 
(and become in fact more transparent) in a framework 
where one simply focuses on relations between 
families of multilinear functionals on a vector space 
$\cV$, without assuming that $\cV$ has a multiplicative 
structure.  In the present subsection we explain how this 
goes.  It will come in handy to use the following notation.

\begin{notation}   \label{def:15}
Let $\cV$ be a vector space over $\bC$.  We denote
\begin{equation}   \label{eqn:15a}
\fMcV := \{ \uPhi : \uPhi = ( \varphi_n )_{n=1}^{\infty}, 
\mbox{ with $\varphi_n : \cV^n \to \bC$ multilinear for 
       every }   n \in \bN \} .
\end{equation}
A family $\uPhi = ( \varphi_n )_{n=1}^{\infty} \in \fMcV$
is said to be {\em tracial} when it has the property that 
\begin{equation}   \label{eqn:15b}
\varphi_n ( x_2, \ldots , x_n, x_1)
= \varphi_n (x_1, \ldots , x_n),
\ \ \forall \, n \geq 2 \mbox{ and } 
x_1, \ldots , x_n \in \cV.
\end{equation}
\end{notation}

Throughout a substantial part of this paper we will work with
\begin{equation}   \label{eqn:1f}
\left\{  \begin{array}{l}
\mbox{ couples $( \cV , \uPhi )$ where 
                 $\uPhi \in \fMcV$, and }             \\
\mbox{ triples $( \cV , \uPhi, \uPsi )$ where 
                 $\uPhi, \uPsi \in \fMcV$. }
\end{array}  \right.
\end{equation}
The motivating example for these structures is the one 
where $\cV$ is a {\em unital algebra} and where 
$\uPhi = ( \varphi_n )_{n=1}^{\infty}$ is completely 
determined by the linear functional $\varphi_1 : \cV \to \bC$
via the formula
\[
\varphi_n (x_1, \ldots , x_n) = \varphi_1 (x_1 \cdots x_n), 
\ \ \forall \, n \geq 1 \mbox{ and } 
x_1, \ldots , x_n \in \cV
\]
(same for $\uPsi = ( \psi_n )_{n=1}^{\infty}$ being completely 
determined by $\psi_1 : \cV \to \bC$ in the case of a triple
$( \cV , \uPhi , \uPsi )$).  If in the motivating example 
we require that $\varphi_1 ( 1_{{ }_{\cV}} ) = 1$, then 
looking at the couple $( \cV , \uPhi )$ is pretty much the 
same as looking at the ncps $( \cV , \varphi_1 )$.  Likewise, 
considering a triple $( \cV , \uPhi , \uPsi )$ in the 
motivating example boils down to looking at
$( \cV , \varphi_1, \psi_1 )$ -- the latter can be either
a C-ncps or an I-ncps, upon making the requirements that 
$\varphi_1 ( 1_{{ }_{\cV}} ) 
= \psi_1 ( 1_{{ }_{\cV}} ) = 1$, and respectively that 
$\varphi_1 ( 1_{{ }_{\cV}} ) = 1$,  
$\psi_1 ( 1_{{ }_{\cV}} ) = 0$.

There were four brands of cumulants appearing in the 
discussion of Sections 1.1 and 1.2:

\vspace{10pt}

\begin{tabular}{ll}
$\bullet$ &
free cumulants $( \kappa_n )_{n=1}^{\infty}$ 
associated to a ncps $( \cA , \varphi )$;       \\

$\bullet$ &
Boolean cumulants $( \beta_n )_{n=1}^{\infty}$ 
associated to a ncps $( \cA , \varphi )$;       \\

$\bullet$ &
c-free cumulants 
$( \kappa_n^{(c)} )_{n=1}^{\infty}$ associated 
to a C-ncps $( \cA , \varphi, \chi )$;           \\

$\bullet$ &
infinitesimal free cumulants 
$( \kappa_n ' )_{n=1}^{\infty}$ associated to
an I-ncps $( \cA , \varphi, \varphi ' )$.
\end{tabular}

\vspace{10pt}

\noindent
The definitions of all these four brands of cumulants go
through, without any change, to the situation where instead 
of an ncps $( \cA , \varphi )$ we consider a couple 
$( \cV , \uPhi )$ as on the first line of (\ref{eqn:1f}), 
and where instead of a C-ncps $( \cA , \varphi, \chi )$ or 
I-ncps $( \cA , \varphi , \varphi ' )$ we consider a triple
$( \cV , \uPhi, \uPsi )$ as on the second line of (\ref{eqn:1f}). 
The precise formulas for all these cumulants are reviewed in 
Section 3 below.  A pleasing feature arising in this more 
general framework is that the resulting families of cumulants 
belong to the same space $\fMcV$ that $\uPhi$ and $\uPsi$ 
were picked from.

We now proceed to explain how the construction of the map 
$\Psi$ of Belinschi-Shlyakhtenko extends to the framework of
(\ref{eqn:1f}), and to present the generalization of 
Theorem \ref{thm:14} to this situation.  We will take as 
input (besides the vector space $\cV$) a linear map 
$\Delta : \cV \to \cV \otimes \cV$, and we will use a 
transformation $\Deltastar$ of the space $\fMcV$ which is 
constructed from $\Delta$ in the way described as follows.

\begin{notation}   \label{def:16}
{\em (The transformation $\Deltastar$.)}

Let $\cV$ be a vector space over $\bC$ and let 
$\Delta : \cV \to \cV \otimes \cV$ be a linear map.

\vspace{6pt}

(1) For every $n \in \bN$ and $m \in \{1,\ldots,n\}$ we let
$\Delta_n^{(m)}: \cV^{\otimes n} \to \cV^{\otimes(n+1)}$ be
the linear map determined by the requirement that 
\[
\Delta_n^{(m)}(x_1\otimes\cdots\otimes x_n)
=x_1\otimes\cdots\otimes x_{m-1} \otimes 
\bigl( \Delta x_m \bigr)
\otimes x_{m+1} \otimes \cdots \otimes x_n,
\]
for all $x_1, \ldots , x_n \in \cV$. 
Note that in particular one has $\Delta = \Delta_1^{(1)}$.

\vspace{6pt}

(2) For every $n\in\bN$ we let
$\Gamma_n : \cV^{\otimes n} \to \cV^{\otimes n}$ 
be the linear map determined by the requirement that 
\[
\Gamma_n (x_1\otimes\cdots\otimes x_n)
= x_2\otimes\cdots\otimes x_n \otimes x_1, 
\mbox{ for all } x_1, \ldots , x_n \in \cV. 
\]
In particular $\Gamma_1$ is the identity map on $\cV$ and 
$\Gamma_2$ is the so-called flip map, 
$\Gamma_2 (x_1 \otimes x_2) = x_2 \otimes x_1$.

\vspace{6pt}

(3) For every $n\in\bN$ we let
$\widetilde{\Delta}_n : \cV^{\otimes n} \to \cV^{\otimes (n+1)}$ 
be the linear map defined by 
\[
\widetilde{\Delta}_n 
= \sum_{m=1}^n \Gamma_{n+1}^m \circ \Delta_n^{(m)}.
\]
[For a concrete illustration, suppose that 
$x_1, \ldots , x_n \in \cV$ are such that every $\Delta x_i$
is a simple tensor $x_i ' \otimes x_i ''$, 
$1 \leq i \leq n$.  Then 
$\widetilde{\Delta}_n (x_1 \otimes \cdots \otimes x_n)$
comes out as 
\[
\sum_{m=1}^n x_m '' \otimes x_{m+1} \otimes \cdots \otimes x_n 
\otimes x_1 \otimes \cdots \otimes x_{m-1} \otimes x_m ' .]
\]

\vspace{6pt}

(4) Let $\fMcV$ be the space of families of multilinear functionals
from Notation \ref{def:15}.  We define a transformation 
$\Deltastar : \fMcV \to \fMcV$ as follows: given 
$\uPhi = ( \varphi_n )_{n=1}^{\infty} \in \fMcV$, we put
\begin{equation}   \label{eqn:16a}
\Deltastar ( \uPhi ) := \uPsi = ( \psi_n )_{n=1}^{\infty}
\end{equation}
where 
\begin{equation}   \label{eqn:16b}
\psi_n := \varphi_{n+1} \circ \widetilde{\Delta}_n, 
\ \ \forall \, n \in \bN .
\end{equation}
Equation (\ref{eqn:16b}) tacitly uses some natural 
identifications: first, $\varphi_{n+1}$ is viewed 
as a linear map $\cV^{\otimes (n+1)} \to \bC$ (rather 
than a multilinear map $\cV^{n+1} \to \bC$); this makes 
$\varphi_{n+1} \circ \widetilde{\Delta}_n$ be defined as
a linear map $\cV^{\otimes n} \to \bC$, which is then
identified with a multilinear map $\cV^n \to \bC$, as
$\psi_n$ is required to be.
\end{notation}

\begin{theorem}   \label{thm:17}
Let $\cV$ be a vector space over $\bC$ and let 
$\Delta : \cV \to \cV \otimes \cV$ be a linear map.
Let $\uChi \in \fMcV$, and let
\begin{equation}   \label{eqn:17a} 
\uPhi ' := \Deltastar ( \uBeta_{\uChi} ), 
\end{equation}
with $\Deltastar$ as above and where $\uBeta_{\uChi} \in \fMcV$
is the family of Boolean cumulants associated to 
$( \cV , \uChi )$.  Then for any tracial 
$\uPhi \in \fMcV$, the following happens: denoting the 
c-free cumulants of $( \cV, \uPhi, \uChi )$ by 
$\uKappa^{(c)}$ and denoting the infinitesimal free cumulants 
of $( \cV, \uPhi, \uPhi ')$  by $\uKappa '$, one has the 
relation
\begin{equation}   \label{eqn:17b} 
\uKappa ' = \Deltastar ( \uKappa^{(c)} \, ).
\end{equation}
\end{theorem}

\vspace{10pt}

\begin{remark}   \label{rem:18}
(1) The generalization ``$\Psi_{ { }_{\Delta} }$'' of the 
map $\Psi_k$ from Section 1.2 is implicitly captured by 
Equation (\ref{eqn:17a}) in Theorem \ref{thm:17}.  The 
explicit formula for $\Psi_{ { }_{\Delta} }$ would simply be
\[
\Psi_{ { }_{\Delta} } ( \uChi ) 
= \Deltastar ( \uBeta_{\uChi} ), 
\ \ \uChi \in \fMcV .
\]

(2) In order to derive Theorem \ref{thm:14} out of 
Theorem \ref{thm:17}, one makes $\cV = \bC^k$ and lets 
$\Delta$ be the linear map defined by the prescription that 
\[
\Delta e_i = e_i \otimes e_i, \ \ 1 \leq i \leq k,
\]
where $e_1, \ldots , e_k$ is a fixed basis in $\bC^k$.  
This prescription leads precisely to the statement of 
Theorem \ref{thm:14}, via the natural identification of 
$\bC \langle X_1 , \ldots , X_k \rangle$ as the tensor algebra 
of $\bC^k$.  The details of how this happens are given in 
Section 7 below.
\end{remark}

$\ $

\subsection{Organization of the paper and some related remarks.}

$\ $

\noindent
We conclude this introduction by explaining how the paper is 
organized, and by giving a few highlights on the content of 
the various sections.  

\vspace{6pt}

Section 2 describes the combinatorial background used 
throughout the paper, which revolves around the lattices 
$NC(n)$ of non-crossing partitions.  We make a brief review 
of some standard facts about the partial order given on 
$NC(n)$ by reverse refinement, and we also discuss two 
other partial order relations on $NC(n)$ (denoted as
``$\ll$'' and ``$\, \sqsubseteq$'') which were studied in 
the more recent research literature, and are relevant for 
the considerations of the present paper.  

\vspace{6pt}

The rest of the paper is essentially divided into two parts.  
The first part is set in the more general framework outlined 
in Section 1.3 above, and consists of Sections 3-6:

-- Section 3 reviews the types of cumulants we will 
work with. 

-- In Section 4 we derive an explicit formula for 
c-free cumulants, which is useful for the proof of 
Theorem \ref{thm:17}.

-- Section 5 is devoted to a parallel discussion of how 
certain lattices of non-crossing partitions with symmetries 
($\NCB (n)$ and $\NCBopp (n)$) can be used in order to approach 
the cumulant functionals relevant to infinitesimally free 
and respectively to c-free probability.  In this section 
we observe a simple formula, which seems to have been overlooked 
up to now, and helps clarifying the connection between c-free 
cumulants and the lattices $\NCBopp (n)$.

-- In Section 6 we obtain the proof of Theorem \ref{thm:17}.

\vspace{6pt}

The final part of the paper consists of Sections 7 and 8,
which go in the framework of multi-variable distributions
considered in Section 1.2 above.  In Section 7 we connect to 
the setting of Section 6 and we observe how Theorem \ref{thm:14} 
can be derived from Theorem \ref{thm:17}.  Then in Section 8 we 
show how, on the other hand, Theorem \ref{thm:14} implies 
Theorems \ref{thm:12} and \ref{thm:13}.

\vspace{1cm}

\section{Combinatorial background}

\subsection{Review of some basic \boldmath{$NC(n)$} combinatorics.}

$\ $

\noindent
The workhorse for the combinatorial study of free independence is 
the family of {\em lattices of non-crossing partitions} $NC(n)$.
We review here some basic terminology related to this, which will 
be used throughout the paper.  For a more detailed introduction to 
the $NC(n)$'s, one can for instance consult Lectures 9 and 10 
of \cite{NiSp2006}. 

$\ $

\begin{definition}   \label{def:21}
(1) Let $n$ be a positive integer and let 
$\pi = \{ V_1 , \ldots , V_k \}$ be a partition of 
$\{ 1, \ldots ,n \}$; that is, $V_1 , \ldots , V_k$ are pairwise 
disjoint non-void sets (called the {\em blocks} of $\pi$) with
$V_1 \cup \cdots \cup V_k$ = $\{ 1, \ldots , n \}$.  
The number $k$ of blocks of $\pi$ will be denoted as $| \pi |$,
and we will occasionally use the notation ``$V \in \pi$'' to mean
that $V$ is one of $V_1, \ldots , V_k$.

We say that $\pi$ is {\em non-crossing} to mean that for every 
$1 \leq i_1 < i_2 < i_3 < i_4 \leq n$ such that $i_1$ is in the 
same block with $i_3$ and $i_2$ is in the same block with $i_4$, 
it necessarily follows that all of $i_1, \ldots , i_4$ are in 
the same block of $\pi$.  

Non-crossing partitions can be naturally depicted with the 
numbers $1, \ldots , n$ drawn either along a line or around 
a circle, as illustrated in Figure 1 below.
\vspace{6pt}

(2) For every $n \in \bN$, we denote by $NC(n)$ the set of 
all non-crossing partitions of $\{ 1, \ldots , n \}$.  This is one
of the many combinatorial structures counted by Catalan numbers:
\begin{equation}   \label{eqn:21a}
| NC(n) | = \Cat_n := \frac{ (2n)! }{ n! (n+1)!}, \ \ \forall 
\, n \in \bN .
\end{equation}   
\end{definition}

$\ $

\begin{definition}  \label{def:22}
Let $n$ be a positive integer.  On $NC(n)$ we consider the partial 
order by {\em reverse refinement}, where for $\pi, \rho \in NC(n)$ 
we put
\begin{equation}   \label{eqn:22a}
( \pi \leq \rho )  \ \ecdef \ \Bigl( \mbox{
every block of $\rho$ is a union of blocks of $\pi$} \Bigr).  
\end{equation}
The partially ordered set $( NC(n), \leq )$ turns out to be a 
lattice.  That is, every $\pi_1 , \pi_2 \in NC(n)$ have a least 
common upper bound, denoted as $\pi_1 \vee \pi_2$, and have a 
greatest common lower bound, denoted as $\pi_1 \wedge \pi_2$.  

We will use the notation $0_n$ for the partition of 
$\{ 1, \ldots , n \}$ into $n$ singleton blocks and the notation 
$1_n$ for the partition of $\{ 1, \ldots , n \}$ into one block.
It is immediate that $0_n, 1_n \in NC(n)$ and that 
$0_n \leq \pi \leq 1_n$ for all $\pi \in NC(n)$.
\end{definition}

\vspace{6pt}

\begin{center}
  \setlength{\unitlength}{0.5cm}
  \begin{picture}(10,3)
  \thicklines
  \put(1,3){\line(0,1){3}}
  \put(1,3){\line(1,0){5}}
  \put(2,4){\line(0,1){2}}
  \put(2,4){\line(1,0){2}}
  \put(3,5){\line(0,1){1}}
  \put(4,4){\line(0,1){2}}
  \put(5,3){\line(0,1){3}}
  \put(6,3){\line(0,1){3}}
  \put(7,3){\line(0,1){3}}
  \put(8,3){\line(0,1){3}}
  \put(8,3){\line(1,0){2}}
  \put(9,4){\line(0,1){2}}
  \put(10,3){\line(0,1){3}}
  \put(0.8,6.2){1}
  \put(1.8,6.2){2}
  \put(2.8,6.2){3}
  \put(3.8,6.2){4}
  \put(4.8,6.2){5}
  \put(5.8,6.2){6}
  \put(6.8,6.2){7}
  \put(7.8,6.2){8}
  \put(8.8,6.2){9}
  \put(9.8,6.2){10}
  \end{picture}
\hspace{0.8cm}
\begin{tikzpicture}[scale=2.2,cap=round,>=latex]
		
		\draw[thick] (0cm,0cm) circle(1cm);
		
		\foreach \x in {18,54,...,342} {
			\filldraw[black] (\x:1cm) circle(0.4pt);
		}
		\draw[black] (90:1cm)--(270:1cm);
		\draw[black] (90:1cm)--(306:1cm);
		\draw[black] (306:1cm)to[out=120,in=90](270:1cm);
		\draw[black] (126:1cm)--(198:1cm);
		\draw[black] (342:1cm)--(54:1cm);

		
		\foreach \x/\xtext in {
			18/3,
			54/2,
			90/1,
			126/10,
			162/9,
			198/8,
			234/7,
			270/6,
			306/5,
			342/4
		}
		\draw (\x:1.25cm) node[] {$\xtext$};
		
\end{tikzpicture}

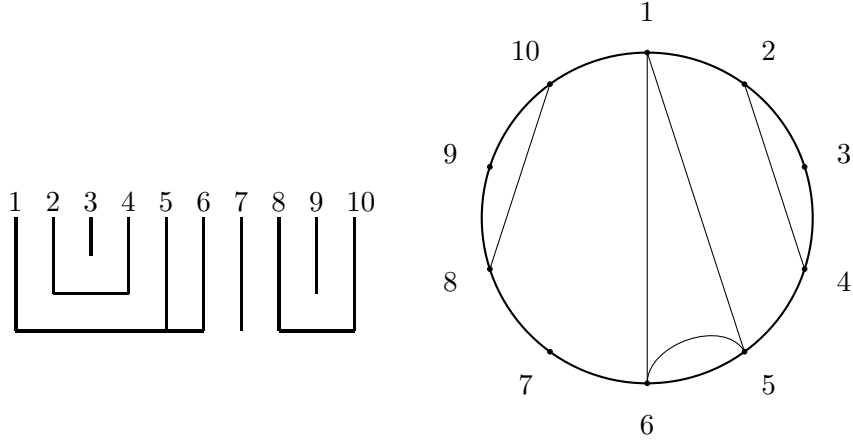
{\bf Figure 1.}  {\em The partition
$\pi = \{ \, \{ 1,5,6 \}, \, \{ 2,4 \}, \, \{ 3 \}, 
\, \{ 7 \}, \, \{ 8,10 \} \, \{ 9 \} \, \}
\in NC(10)$,}

{\em in linear representation (left) and in circular 
representation (right).}
\end{center}

$\ $

\begin{defrem}   \label{def:23}
{\em (Kreweras complementation map.) }

\noindent
One has a very useful order-reversing bijection 
$K_n : NC(n) \to NC(n)$, called the 
{\em Kreweras complementation map}, which provides an 
anti-isomorphism of the lattice $( NC(n), \leq )$.
The pictorial description of $K_n$ is given by using partitions of 
the set $\{ 1, \ldots , 2n \}$, as follows.  

\vspace{6pt}

$\bullet$ For every $\pi, \sigma \in NC(n)$, let us denote by 
$\piodd \sqcup \sigmaeven$ the partition of $\{ 1, \ldots , 2n \}$ 
which is obtained when we make $\pi$ become a partition of 
$\{ 1,3 , \ldots , 2n-1 \}$ and we make $\sigma$ become a partition 
of $\{ 2, 4, \ldots , 2n \}$, in the natural way.  That is, 
$\piodd \sqcup \sigmaeven$ consists of blocks of the form 
$\{ 2v -1 : v \in V \}$ where $V \in \pi$, and of blocks of the 
form $\{ 2w : w \in W \}$ where $W \in \sigma$.

\vspace{6pt}

$\bullet$ For $\pi, \sigma \in NC(n)$ it is not generally true that
$\piodd \sqcup \sigmaeven$ is a non-crossing partition of 
$\{ 1, \ldots , 2n \}$.  If we fix $\pi \in NC(n)$, then the set
\[
\{ \sigma \in NC(n) : \piodd \sqcup \sigmaeven \in NC(2n) \}
\]
turns out to have a largest element $\sigma_{\mathrm{max}}$ with 
respect to reverse refinement order.  The Kreweras complement 
$K_n ( \pi )$ is, by definition,
this special partition $\sigma_{\mathrm{max}}$.

\vspace{6pt}

\noindent
[As a concrete example: for the partition $\pi$ depicted in 
Figure 1, one finds that $K_{10} ( \pi )$ is the partition 
$\bigl\{ \, \{ 1,4 \}, \, \{ 2,3 \}, \, \{ 5 \}, \, 
\{ 6,7,10 \}, \, \{ 8,9 \}, \, \bigr\} \in NC(10)$.

\vspace{6pt}

It is easily verified (see e.g. pages 146-148 in Lecture 9 
of \cite{NiSp2006}) that the map $\pi \mapsto K_n ( \pi )$ 
described above gives indeed an anti-isomorphism from 
$( NC(n) , \leq )$ to itself.
\end{defrem}

$\ $

\begin{defrem}   \label{def:24}
{\em (M\"obius function.) }

\noindent
For every $n \in \bN$, we will use the notation
$\Moeb_n : \{  ( \pi , \rho ) \in NC(n)^2 : \pi \leq \rho \} \to \bZ$
for the M\"obius function of the lattice $(NC(n) , \leq )$, defined 
via the general M\"obius function machinery used for any finite 
partially ordered set (see e.g. Chapter 3 of the monograph
\cite{St1997}).  One can explicitly write $\Moeb_n ( \pi , \rho )$ as a 
product of signed Catalan numbers (see e.g. \cite{NiSp2006}, pp. 162-167 
in Lecture 10); it is useful for what follows to record that one has in 
particular the formula
\begin{equation}  \label{eqn:24a}
\Moeb_n ( \pi , 1_n) = \Moeb_n ( 0_n , K_n ( \pi ))
=(-1)^{| \pi |-1} \prod_{V \in K_n ( \pi )} \Cat_{|V|-1},
\end{equation}
holding for every $\pi \in NC(n)$ (where the Catalan numbers 
$( \Cat_m )_{m=1}^{\infty}$ are as reviewed in Definition 
\ref{def:21}, and we also make the convention to put $\Cat_0 := 1$).
\end{defrem}

$\ $

For the subsequent considerations related to the framework of 
c-free independence, it is also useful to record the following 
facts about how blocks of a non-crossing partitions are nested 
inside each other.

\begin{defrem}   \label{def:25}
Let $n$ be a positive integer and let $\pi$ be in $NC(n)$.  

\vspace{6pt}

(1) Let $V,W$ be two blocks of $\pi$.  We say that $V$ is 
{\em nested} inside $W$ to mean that one has
\begin{equation}   \label{eqn:25a}
\min (W) < \min (V) \leq \max (V) < \max (W).
\end{equation}
Due to the non-crossing property of $\pi$, it is immediate 
that the condition (\ref{eqn:25a}) is equivalent to the 
apparently weaker requirement that:
\begin{equation}   \label{eqn:25b}
V \neq W, \mbox{ and } \exists \, v \in V \mbox{ such that }
\min (W) < v < \max (W).
\end{equation}

\vspace{6pt}

(2) A block $V$ of $\pi$ is said to be {\em inner} if there 
exists a block $W$ such that $V$ is nested inside $W$.  In 
the opposite case, $V$ is said to be an {\em outer} block
of $\pi$.

\vspace{6pt}

(3) Let $V$ be an inner block of $\pi$.  It is easy to check
(see e.g. Proposition 2.10 in \cite{BeNi2008}) that there exists 
a block $W$ of $\pi$, uniquely determined, with the properties 
that: 

\noindent
(i) $V$ is nested inside $W$, and

\noindent
(ii)  there is no $W' \in \pi$ such that $V$ is nested 
inside $W'$ and $W'$ is nested inside $W$.

\noindent 
We will refer to this $W$ as
the {\em parent-block} of $V$ in $\pi$.
\end{defrem}

$\ $

\subsection{The partial order relations \boldmath{$\ll$} 
and \boldmath{$\sqsubseteq$}.}

$\ $

\noindent
In this paper we also make use of two other partial order 
relations on $NC(n)$, both of them coarser than reverse refinement, 
which are denoted as ``$\ll$'' and as ``$\sqsubseteq$''.  The 
partial order $\ll$ has been used for some time in free 
probability (starting with \cite{BeNi2008}), in the description 
of relations between free and Boolean cumulants.  The partial 
order $\sqsubseteq$ is in a certain sense dual to $\ll$; it was 
introduced in \cite{JV2015}, and its role was thoroughly investigated 
in the recent paper \cite{BiJV2018}, in a more general setting which 
refers to the Bruhat order on Coxeter groups.

$\ $

\begin{defrem}      \label{def:26}
{\em (The partial order ``$\ll$''.)}

\noindent
(1) For $\pi , \rho \in NC(n)$, we will write $\pi \ll \rho$ to mean 
that $\pi \leq \rho$ and that, in addition, for every block $W$ of 
$\rho$ there exists a block $V$ of $\pi$ such that $\min (W), \max (W) \in V$.  

\vspace{6pt}

\noindent
(2) Since in this paper we have a lot of occurrences of the special case
``$\pi \ll 1_n$'', let us record the obvious fact that this simply amounts 
to requiring $\pi$ to have a unique outer block $W$, with $1,n \in W$.
Another immediate fact which is relevant for what follows is that the 
Kreweras complementation map $K_n$ maps the set 
$\{ \pi \in NC(n) : \pi \ll 1_n \}$ onto 
$\{ \sigma \in NC(n) : \{ n \} \mbox{ is a singleton block of $\sigma$} \}$.

\vspace{6pt}

\noindent
(3) In the discussion around $\ll$, a special role is played by 
interval partitions.  A partition $\pi \in NC(n)$ is said to be
an {\em interval partition} when every block $V$ of $\pi$ is of the form
$V = [i,j] \cap \bN$ for some $1 \leq i \leq j \leq n$.  The set of all 
interval partitions of $\{ 1, \ldots , n \}$ will be denoted as 
$\Int (n)$.  It is immediate that $\Int (n)$ is precisely equal to the 
set of maximal elements of the poset $( NC(n), \ll )$.
\end{defrem}

$\ $

\begin{remark}   \label{rem:27}
It will be of relevance for what follows to have some information
about the structure of lower and of upper ideals of the poset 
$( NC(n) , \ll )$.  We review this here, following \cite{BeNi2008}. 

\vspace{6pt}

$\bullet$ {\em Lower ideals.}
Let $\rho$ be a fixed partition in $NC(n)$.  For every block $W \in \rho$
such that $|W| \geq 3$, let us split $W$ into the doubleton block
$\{ \min (W), \max (W) \}$ and $| W | - 2$ singleton blocks; when 
doing this, we obtain a partition $\rho_0 \leq \rho$ in $NC(n)$, 
such that all the blocks of $\rho_0$ have either 1 or 2 elements. From 
Definition \ref{def:26} it is immediate that for $\pi \in NC(n)$ we 
have: $\pi \ll \rho \ \Leftrightarrow \ \rho_0 \leq \pi \leq \rho$.
Thus the lower ideal $\{ \pi \in NC(n) \ : \ \pi \ll \rho \}$
is just the interval $[ \rho_0 , \rho ]$ with respect to reverse
refinement order, for which one has a nice structure 
theorem (as presented for instance in \cite{NiSp2006}, pages 148-150 in 
Lecture 9); in particular, one can explicitly write the cardinality of 
this lower ideal, which is
\begin{equation}   \label{eqn:27a}
| \, \{ \pi \in NC(n) \ : \ \pi \ll \rho \} \, | = 
\prod_{W\in \rho}  \Cat_{|W|-1}.
\end{equation}

\vspace{6pt}

$\bullet$ {\em Upper ideals.}
For a fixed $\pi \in NC(n)$, it turns out that the upper ideal 
$\{ \rho \in NC(n): \pi \ll \rho \}$ can be identified to a Boolean 
lattice, with the rank of any $\rho$ in this lattice being equal 
to $|\pi| - | \rho |$ (see Proposition 2.13 and Remark 2.14 in 
\cite{BeNi2008}).  Clearly, the Boolean lattice 
$\bigl( \, \{ \rho \in NC(n): \pi \ll \rho \} \, , \, \ll \, \bigr)$ 
is trivial if and only if $\pi $ is maximal with respect to $\ll$
in $NC(n)$, that is, if and only if $\pi \in \Int (n)$.
We record here that, as a consequence of the above, one has
\begin{equation}   \label{eqn:27b}
\sum_{  \begin{array}{c}
{\scriptstyle \rho \in NC(n) \ such}   \\
{\scriptstyle that \ \pi \ll \rho}  
\end{array}  } \, (-1)^{| \pi | - | \rho |} 
= \left\{
\begin{array}{ll}
1, &  \text{ if } \pi \in \Int(n) \\
0, &  \text{ otherwise}.
\end{array}\right.
\end{equation} 
\end{remark}

$\ $

\begin{defrem}   \label{def:28}
{\em (The partial order ``$\, \sqsubseteq$''.)}

\noindent
The interval refinement order $\sqsubseteq $ was introduced by 
Josuat-Verg\`es \cite{JV2015}: for $\pi , \rho \in NC(n)$, one writes
$\pi \sqsubseteq \rho$ when $\pi \leq \rho$ and when, in addition, 
the partition induced by $\pi $ on every block of $\rho$ is an interval 
partition. 

One has a form of duality between the partial orders $\sqsubseteq$ 
and $\ll$, implemented by the Kreweras complementation map.   This was 
thoroughly investigated in a recent paper by Biane and Josuat-Verg\'es 
\cite{BiJV2018}, in a more general setting involving the Bruhat order on 
Coxeter groups.  (See Section 4 of \cite{BiJV2018}, particularly 
Propositions 4.1 and 4.9.)  For the reader's convenience, we provide in 
Lemma \ref{lem:210} below a self-contained proof of an instance of this 
duality which will be needed in the Section 4 of the present paper.  In 
the proof of Lemma \ref{lem:210} we will use some natural operations 
``cut/attach'' on non-crossing partitions, defined as follows.
\end{defrem}

\begin{definition}\label{def:29}
Let $\pi\in NC(n)$ and let $i \in \{1,\ldots, n\}$ belong to an inner 
block $P$ of $\pi$.  

\vspace{6pt}

(1) If $i$ is not the maximal element of $P$, then we define 
$\cut(\pi,i)\in NC(n)$ to be the partition obtained from $\pi$ 
when we cut the block $P$ into two interval pieces in order to 
get that 
\[
Q\in \cut(\pi,i)\iff (Q\in\pi\mbox{ and }Q\not=P)\mbox{ or } 
(Q=\{j\in P: j\le i\})\mbox{ or } (Q=\{j\in P: j>i\}).
\]

\vspace{6pt}

(2) We define $\attach(\pi,i)\in NC(n)$ to be the partition obtained from 
$\pi$ when we attach the block $P$ to its parent-block $R$ (defined in the 
way described in Definition \ref{def:25}).  That is, we have 
\[
Q\in \attach (\pi,i) 
\iff  (Q\in\pi\mbox{ and }Q\not\in\{P,R\})\mbox{ or } (Q=P\cup R).
\] 
\end{definition}

\begin{lemma} \label{lem:210} 
The restriction of the Kreweras complementation map $K_n$
gives an anti-isomorphism between the posets
\begin{equation}   \label{eqn:210a}
\left( \left\{\pi\in NC(n) : \pi\ll 1_n\right \} \, , 
\, \sqsubseteq \right)
\mbox{ and } 
\left( \left\{\sigma\in NC(n): \{n\}\in\sigma\right\} \, ,
\, \ll \right).
\end{equation}
\end{lemma}

\begin{proof}  
Consider the Hasse diagrams of the two posets indicated in the lemma
(where recall that for any finite poset $( \cP , \prec )$, the corresponding 
Hasse diagram is the graph with vertex-set $\cP$ and with a ``down-edge''
going from $p$ to $p'$ if and only if $p,p'$ are distinct elements of $\cP$
such that $p' \prec p$ and such that there is no element 
$q \in \cP \setminus \{ p, p' \}$ with $p' \prec q \prec p$).
It is easy to verify that the down-edges in these two Hasse diagrams are 
described as follows.

\vspace{6pt}

$\bullet$ In the Hasse diagram of 
$\displaystyle{\left(\left\{\pi\in NC(n) : \pi \ll 1_n\right \}, 
\sqsubseteq \right)}$, 
there is a down edge from $\pi$ to $\pi'$ if and only if for some $i$ 
in an inner block of $\pi$ and not maximal in that block we have that 
$\pi'= \cut(\pi,i)$. 

\vspace{6pt}

$\bullet$ In the Hasse diagram of 
$\displaystyle{\left(\left\{\sigma\in NC(n): \{n\}\in\sigma\right\},
\ll \right)}$, 
there is a down edge from $\rho$ to $\rho'$ if and only if for some $j$ 
in an inner block of $\rho '$ we have $\rho = \attach( \rho',j)$. 

\vspace{6pt}

We have already noticed in Remark \ref{def:26}(2) that $K_n$ provides a 
bijection between the underlying sets of the two posets indicated in 
(\ref{eqn:210a}).  In addition to that, let us make the following 
elementary observation, which holds for any partition $\pi \in NC(n)$ 
with unique outer block: if a number $i \in \{ 1, \ldots , n \}$ belongs 
to an inner block of $\pi$ and is not the maximal element of that block, 
then $i$ belongs to an inner block of $K_n (\pi)$ and one has 
\begin{equation}   \label{eqn:210b}
K_n (\cut(\pi, i)) = \attach( K_n (\pi),i).
\end{equation}
Upon combining (\ref{eqn:210b}) with the explicit descriptions recorded 
above for the edges of the two Hasse diagrams, one finds that $K_n$
reverses the edges in these Hasse diagrams; this implies that $K_n$ is 
indeed a poset anti-isomorphism, as required. 
\end{proof}

$\ $

\section{Review of four types of cumulant functionals}

Throughout this section we fix a vector space $\cV$ over $\bC$, 
and we will work with the space $\fMcV$ of families of multilinear 
functionals introduced in Notation \ref{def:15}.  The goal 
of the section is to review the definitions of four types of 
cumulant functionals which we use later in the paper,
in the framework of a couple $( \cV , \uPhi )$ with 
$\uPhi \in \fMcV$ or in the framework of a triple 
$( \cV , \uPhi, \uPsi )$ with $\uPhi , \uPsi \in \fMcV$.  
All the formulas we will use for definitions are very familiar
to the people working in the area; these definitions are usually 
stated in the case (mentioned as ``motivating example'' in 
Section 1.3) when $\cV$ is a unital algebra -- however, nothing 
changes when we move to the somewhat more general framework chosen 
here.

Before starting, we record a customary notation which will 
appear in the formulas for all four types of cumulants:
given an $n \in \bN$, a tuple
$( x_1, \ldots , x_n ) \in \cV^n$, and a non-empty subset
$M = \{ i_1, \ldots , i_m \} \subseteq \{ 1, \ldots , n \}$
with $i_1 < \cdots < i_m$, we denote
\begin{equation}  \label{eqn:3a}
( x_1, \ldots , x_n ) \mid M 
:= ( x_{i_1}, \ldots , x_{i_m} ) \in \cV^m.
\end{equation}

$\ $

\subsection{Review of free cumulants associated 
to a couple \boldmath{$( \cV , \uPhi )$}.}

$\ $

\noindent
Consider a couple $( \cV , \uPhi )$, where 
$\uPhi = ( \varphi_n )_{n=1}^{\infty} \in \fMcV$.  The 
{\em free cumulants} associated to $( \cV , \uPhi )$ are
the family of multilinear functionals 
$\uKappa = ( \kappa_n )_{n=1}^{\infty} \in \fMcV$ which 
is uniquely determined by the requirement that
\begin{equation}   \label{eqn:31a}
\varphi_n (x_1, \ldots , x_n)  = 
\sum_{\pi \in NC(n)} \prod_{V \in \pi} 
\kappa_{|V|} ( \, (x_1, \ldots , x_n) \mid V \, ), 
\end{equation}
holding for all $n \in \bN$ and $x_1, \ldots , x_n \in \cV$.
The family of Equations (\ref{eqn:31a}) can be solved in 
order to give explicit formulas for the $\kappa_n$'s, which 
come out as follows:
\begin{equation}   \label{eqn:31b}
\kappa_n (x_1, \ldots , x_n) =
\sum_{\pi \in NC(n)} \prod_{V \in \pi} 
\Moeb_n ( \pi , 1_n )
\, \varphi_{|V|} ( \, (x_1, \ldots , x_n) \mid V \, ),  
\end{equation}
for all $n \in \bN$ and $x_1, \ldots , x_n \in \cV$,
where $\Moeb_n ( \pi , 1_n )$ is the value of the M\"obius 
function of $NC(n)$ which was reviewed in Remark \ref{def:24}.

$\ $

\subsection{Review of Boolean cumulants associated
to a couple \boldmath{$( \cV , \uPhi )$}.}

$\ $

\noindent
Consider a couple $( \cV , \uPhi )$, where 
$\uPhi = ( \varphi_n )_{n=1}^{\infty} \in \fMcV$.  The 
{\em Boolean cumulants} associated to $( \cV , \uPhi )$ 
are the family of multilinear functionals 
$\uBeta = ( \beta_n )_{n=1}^{\infty} \in \fMcV$ which 
is uniquely determined by the requirement that
\begin{equation}   \label{eqn:32a}
\varphi_n (x_1, \ldots , x_n) = 
\sum_{\pi \in \Int (n)} \prod_{V \in \pi} 
\beta_{|V|} ( \, (x_1, \ldots , x_n) \mid V \, ),   
\end{equation}
holding for all $n \in \bN$ and $x_1, \ldots , x_n \in \cV$,
where $\Int (n)$ is the set of interval partitions
reviewed in Definition \ref{def:26}(3).
The family of Equations (\ref{eqn:32a}) can be solved 
in order to give explicit formulas for the $\beta_n$'s, 
which come out as follows:
\begin{equation}   \label{eqn:32b}
\beta_n (x_1, \ldots , x_n) = 
\sum_{\pi \in \Int (n)} \prod_{V \in \pi} 
(-1)^{ | \pi | + 1 }
\varphi_{|V|} ( \, (x_1, \ldots , x_n) \mid V \, ),   
\end{equation}
for all $n \in \bN$ and $x_1, \ldots , x_n \in \cV$.
This is analogous to how (\ref{eqn:31a}) was solved
in order to obtain (\ref{eqn:31b}), with an additional 
simplification due to the fact that the M\"obius function 
of $( \Int (n), \leq )$ only takes the values $\pm 1$. 

$\ $

\subsection{Review of infinitesimal free cumulants
associated to a triple \boldmath{$( \cV , \uPhi, \uPhi ')$}.}

$\ $

\noindent
Consider a triple $( \cV , \uPhi, \uPhi ' )$, where 
$\uPhi = ( \varphi_n )_{n=1}^{\infty}$ and 
$\uPhi ' = ( \varphi_n ')_{n=1}^{\infty}$ are in 
$\fMcV$.  The {\em infinitesimal free cumulants} 
associated to $( \cV , \uPhi, \uPhi ' )$ are the 
family of multilinear functionals 
$\uKappa ' = ( \kappa_n ' )_{n=1}^{\infty}$ 
which is uniquely determined by the requirement that
\begin{equation}   \label{eqn:34a}
\varphi_n ' (x_1 , \ldots , x_n)  = 
\end{equation}
\[
\sum_{\pi \in NC(n)} \sum_{V_o \in \pi} 
\kappa_{ | V_o | } '
\bigl( \, (x_1, \ldots , x_n) \mid V_o \bigr)
\cdot  \prod_{  \begin{array}{c}
{\scriptstyle V \in \pi}  \\
{\scriptstyle V \neq V_o}
\end{array} } 
\ \kappa_{ |V|; \uPhi } ( \, (x_1, \ldots , x_n) \mid V ), \\
\]                              
holding for all $n \in \bN$ and $x_1, \ldots , x_n \in \cV$,
where $\uKappa_{\uPhi} =
( \kappa_{n ; \uPhi} )_{n=1}^{\infty}$ 
are the free cumulants associated to $( \cV , \uPhi )$.
The family of Equations (\ref{eqn:34a}) can be solved in order 
to give an explicit formula for $\kappa_n ' (x_1, \ldots , x_n)$, 
which comes out as follows:
\begin{equation}    \label{eqn:34b}
\kappa_n ' (x_1 , \ldots , x_n) = 
\end{equation}
\[
\sum_{\pi \in NC(n)}  \sum_{V_o \in \pi}
\ \Moeb_n ( \pi , 1_n)
\cdot \varphi_{ | V_o | } '
\bigl( \, (x_1, \ldots , x_n) \mid V_o \bigr)
\cdot  \prod_{  \begin{array}{c}
{\scriptstyle V \in \pi}  \\
{\scriptstyle V \neq V_o}
\end{array} } 
\ \varphi_{ |V| } ( \, (x_1, \ldots , x_n) \mid V \bigr),
\]
for all $n \in \bN$ and $x_1, \ldots , x_n \in \cV$.
Here $\Moeb_n ( \pi , 1_n )$ is the same value of the M\"obius 
function on $NC(n)$ which appeared in Equation (\ref{eqn:31b})
above.  

A way to remember the formulas (\ref{eqn:34a}) and 
(\ref{eqn:34b}) is by noting that they are precisely what 
comes out when one performs a ``formal derivative with 
respect to $\varphi$'' in the formulas (\ref{eqn:31a}) and 
(\ref{eqn:31b}) concerning the free cumulant functionals of 
$( \cV , \uPhi )$.  Another way of thinking about Equations
(\ref{eqn:34a}) and (\ref{eqn:34b}) goes by treating the 
double sums on their right-hand sides as single sums over 
certain sets of non-crossing partitions ``of type B'' (cf.
\cite{FeNi2010}, also \cite{Fe2012}; a brief review of this 
point of view is shown in Section 5.1 below).

$\ $

\subsection{Review of c-free cumulants
associated to a triple \boldmath{$( \cV , \uPhi, \uChi )$}.}

$\ $

\noindent
Consider a triple $( \cV , \uPhi, \uChi )$, where 
$\uPhi = ( \varphi_n )_{n=1}^{\infty}$ and 
$\uChi = ( \chi_n )_{n=1}^{\infty}$ are in $\fMcV$.  
The {\em c-free cumulants} associated to 
$( \cV , \uPhi, \uChi )$ are the family of 
multilinear functionals 
$\uKappa^{(c)} = ( \kappa^{(c)}_n )_{n=1}^{\infty}$ 
which is uniquely determined by the requirement that
\begin{equation}   \label{eqn:33a}
\chi_n (x_1 , \ldots , x_n) = 
\end{equation}
\[
\sum_{\pi \in NC(n)}\prod_{  \begin{array}{c}
{\scriptstyle V \in \pi}  \\
{\scriptstyle V \, \text{inner}}
\end{array} } \kappa_{ |V| ; \uPhi } 
\bigl( \, (x_1, \ldots , x_n) \mid V \bigr)
\prod_{  \begin{array}{c}
{\scriptstyle W \in \pi}  \\
{\scriptstyle W \, \text{outer}}
\end{array} } \kappa^{(c)}_{ |W| } 
\bigl( \, (x_1, \ldots , x_n) \mid W \bigr), 
\]                         
holding for all $n \in \bN$ and 
$x_1, \ldots , x_n \in \cV$, where 
$\uKappa_{\varphi} = ( \kappa_{n ; \uPhi} )_{n=1}^{\infty}$ 
are the free cumulants associated to $( \cV , \uPhi )$.

The fact that $\uKappa^{(c)}$ can indeed be defined 
by using Equation (\ref{eqn:33a}) is easily seen when
one isolates the term $\kappa_n^{(c)} (x_1, \ldots , x_n)$
indexed by the partition $1_n \in NC(n)$ on the right-hand 
side of (\ref{eqn:33a}); a recursive argument then shows 
that all the values $\kappa_n^{(c)} (x_1, \ldots , x_n)$ 
are uniquely determined in terms the functionals in
$\uChi$ and $\uKappa_{\uPhi}$ (where the latter ones 
can be calculated from knowing $\uPhi$).  This way of 
defining c-free cumulants goes back to \cite{BoLeSp1996}, 
and is the one commonly used in the research literature 
on c-freeness.

It is not so straightforward to find some nicely structured 
formulas which give the $\kappa_n^{(c)}$'s explicitly, in 
terms of $\uPhi$ and $\uChi$ (analogously to what we had in 
Equations (\ref{eqn:31b}), (\ref{eqn:32b}) and (\ref{eqn:34b}) 
of the preceding subsections).  An interesting aspect of the 
c-free theory is that one can still develop an approach based 
on certain lattices of partitions ``of type $B-opp$'', 
found in \cite{Ca2000}, only that
the cumulant functionals resulting from that approach don't 
coincide with the customary $\kappa_n^{(c)}$'s from 
(\ref{eqn:33a}).  We will elaborate on this point in 
Section 5 below.  Before that, in Section 4 we will present 
a direct derivation of an explicit formula relevant for the 
proof of our Theorem \ref{thm:17}, which expresses 
$\kappa_n^{(c)}$ in terms of $\uPhi$ and 
$\uBeta_{\uChi}$ ( = the family of Boolean cumulants
associated to $( \cV , \uChi )$).

$\ $

\section{An explicit formula for c-free cumulants}

Throughout this section we fix a vector space $\cV$
and two families of functionals 
$\uPhi = ( \varphi_n )_{n=1}^{\infty}$ and
$\uChi = ( \chi_n )_{n=1}^{\infty}$
picked from the space $\fMcV$ of Notation \ref{def:15}.
We consider the family of c-free cumulants 
$\uKappa^{(c)} = ( \kappa_n^{(c)} )_{n=1}^{\infty}
\in \fMcV$ associated to the triple 
$( \cV , \uPhi, \uChi )$ in the way described in 
Section 3.4.  In the present section we put into evidence an 
explicit formula for $\kappa_n^{(c)} (x_1, \ldots , x_n)$, 
stated in the next proposition.

\begin{proposition} \label{prop:41}
Consider the notations introduced above, and let us also
consider the family 
$\uBeta_{\uChi} = ( \beta_{n ; \uChi} )_{n=1}^{\infty}$
of Boolean cumulants associated to $( \cV , \uChi )$.
For every $n \in \bN$ and $x_1, \ldots , x_n \in \cV$ 
one has:
\begin{equation}    \label{eqn:41a}
\kappa_n^{(c)} (x_1, \ldots , x_n)  = 
\end{equation}
\[
\sum_{ \begin{array}{c}
{\scriptstyle \pi \in NC(n),}  \\
{\scriptstyle \pi \ll 1_n}
\end{array} } \ \Moeb_n ( \pi , 1_n)
\cdot \beta_{ | V_o ( \pi ) | ; \uChi }
\bigl( \, (x_1, \ldots , x_n) \mid V_o ( \pi ) \bigr)
\cdot  \prod_{  \begin{array}{c}
{\scriptstyle V \in \pi}  \\
{\scriptstyle V \neq V_o ( \pi )}
\end{array} } \ \varphi_{ |V| } 
\bigl( \, (x_1, \ldots , x_n) \mid V \bigr),
\]
where (following the notations from Section 2) we write 
``$\pi \ll 1_n$'' to mean that $\pi$ has a unique outer 
block, and where for $\pi \ll 1_n$ we denote the unique 
outer block of $\pi$ as $V_o ( \pi )$.
\end{proposition}

Proposition \ref{prop:41} solves the implicit equations 
(\ref{eqn:33a}) of the preceding section in a somewhat 
non-canonical way, which is however what is needed for 
going towards the connection with infinitesimal free 
cumulants.  (As the reader may notice, the expression 
on the right-hand side of (\ref{eqn:41a}) bears some 
resemblance with the formula for free infinitesimal 
cumulants reviewed in Equation (\ref{eqn:34b}), with
the Boolean cumulants of $\uChi$ appearing in the place 
where we would want to have occurrences of $\uPhi '$.)

In order to prove Proposition \ref{prop:41}, we will use
several lemmas.

\begin{lemma}   \label{lem:42}
Let $\uKappa_{\uPhi} 
     = ( \kappa_{n ; \uPhi} )_{n=1}^{\infty}$ and 
$\uBeta_{\uPhi} = ( \beta_{n ; \uPhi} )_{n=1}^{\infty}$
denote the free and respectively the Boolean cumulants 
associated to $( \cV , \uPhi )$.  Then for every 
$n \in \bN$ and $x_1, \ldots , x_n \in \cV$ one has
\begin{equation}  \label{eqn:42a}
\kappa_{n; \uPhi} (x_1, \ldots , x_n)
= \sum_{ \begin{array}{c}
{\scriptstyle \pi \in NC(n),}  \\
{\scriptstyle \pi \ll 1_n}
\end{array} } \ (-1)^{  | \pi | - 1 }
\prod_{V \in \pi} \beta_{ |V| ; \uPhi } 
( \, (x_1, \ldots , x_n) \mid V ).
\end{equation}
\end{lemma}

\begin{proof}
This connection between free and Boolean cumulants is well-known.
The derivation of Equation (\ref{eqn:42a}) can for instance 
be obtained by an immediate adaptation of the argument proving 
Proposition 3.9 of \cite{BeNi2008}.
\end{proof}

In a somewhat different formulation, 
Equation (\ref{eqn:43a}) of the next lemma was obtained in 

\noindent
\cite[Eqn.(61) in Section 6]{EFPa2018}, as an application 
of the shuffle algebra approach to c-free cumulants 
developed in that paper.  We present here a direct proof
of Equation (\ref{eqn:43a}), relying on some basic 
properties of the partial order $\ll$ on $NC(n)$.

\begin{lemma}     \label{lem:43}
Consider the framework and notations introduced above.  For 
every $n \in \bN$ and $x_1, \ldots , x_n \in \cV$, one has:
\begin{equation}    \label{eqn:43a}
\kappa^{(c)}_n (x_1, \ldots , x_n) =
\end{equation}
\[
\sum_{ \begin{array}{c}
{\scriptstyle \pi \in NC(n),}  \\
{\scriptstyle \pi \ll 1_n}
\end{array} } \ (-1)^{ | \pi | - 1 }
\cdot \beta_{ | V_o ( \pi ) | ; \uChi }
( \, (x_1, \ldots , x_n) \mid V_o )
\cdot  \prod_{  \begin{array}{c}
{\scriptstyle V \in \pi}  \\
{\scriptstyle V \neq V_o ( \pi )}
\end{array} } \ \beta_{ |V| ; \uPhi }
( \, (x_1, \ldots , x_n) \mid V ).
\]
\end{lemma}

\begin{proof}
For every $n \in \bN$ and $x_1, \ldots , x_n \in \cV$, we denote the
right-hand side of \eqref{eqn:43a} as $\lambda_n (x_1, \ldots , x_n)$. 
In this way we define a family of multilinear functionals 
$\uLambda = ( \lambda_n )_{n=1}^{\infty} \in \fMcV$, and the statement of
the lemma amounts to proving that $\uLambda = \uKappa^{(c)}$.  In order to
obtain this equality, it suffices to prove that $\uLambda$ satisfies the 
family of equations (\ref{eqn:33a}) which were used to define $\uKappa^{(c)}$.
That is, we have to verify the family of equalities of the form
\begin{equation}    \label{eqn:43b}
\chi_n (x_1, \ldots , x_n) =
\end{equation}
\[
\sum_{\pi \in NC(n)}  \prod_{  \begin{array}{c}
{\scriptstyle V \in \pi}  \\
{\scriptstyle V \, \text{inner}}
\end{array} } \kappa_{ |V| ; \uPhi } 
( \, (x_1, \ldots , x_n) \mid V )
\prod_{  \begin{array}{c}
{\scriptstyle W \in \pi}  \\
{\scriptstyle W \, \text{outer}}
\end{array} } \lambda_{ |W| } 
( \, (x_1, \ldots , x_n) \mid W ).
\]
For the remaining part of the proof we fix an $n \in \bN$ and some 
$x_1, \ldots x_n \in \cV$, for which we will verify that (\ref{eqn:43b})
holds.  

We start from the sum on the right-hand side of (\ref{eqn:43b}). Let us 
fix for the moment a $\pi \in NC(n)$, and let us focus on the term 
indexed by $\pi$ in the said sum.  This term can be itself written as 
a sum, if we take the following steps:
\begin{equation}   \label{eqn:43c}
\left\{  \begin{array}{l}
\mbox{-- for every inner block $V$ of $\pi$, we replace
      $\kappa_{|V|; \uPhi} ( \, (x_1, \ldots , x_n) \mid V) \, )$ as a}     \\
\mbox{ sum $\sum_V$ indexed by 
      $\{ \rho_V \in NC(|V|) : \rho_V$ has unique outer block$\}$, by }     \\
\mbox{ using Lemma \ref{lem:42};}                                           \\
                                                                            \\
\mbox{-- for every outer block $W$ of $\pi$, we replace
      $\lambda_{|W|} ( \, (x_1, \ldots , x_n) \mid W) \, )$ as}             \\
\mbox{ a sum $\sum_W$ indexed by 
      $\{ \rho_W \in NC(|W|) : \rho_W$ has unique outer block$\}$,}         \\
\mbox{ by using the definition of $\lambda_{|W|}$;}                         \\
                                                                            \\
\mbox{-- we cross-multiply all the sums $\sum_V$ and $\sum_W$ found
       in the preceding }                                               \\
\mbox{ two steps.}
\end{array}  \right.      
\end{equation}
If one chooses a partition $\rho_V$ for every inner block $V \in \pi$ 
and a partition $\rho_W$ for every outer block $W \in \pi$ in the way 
indicated in (\ref{eqn:43c}), then putting all these $\rho_V, \rho_W$ 
together results in a partition $\rho \in NC(n)$ such that 
$\rho \ll \pi$, with ``$\ll$'' being the partial order discussed in 
Remark \ref{def:26}.  We leave it as an exercise to the reader to write down 
the general term of the summation over 
$\{ \rho \in NC(n) : \rho \ll \pi \}$ which is produced by the 
third step of (\ref{eqn:43c}), and to conclude that what one gets is 
the following formula:
\begin{equation}    \label{eqn:43d}
\prod_{  \begin{array}{c}
{\scriptstyle V \in \pi}  \\
{\scriptstyle V \, \text{inner}}
\end{array} } \kappa_{ |V| ; \uPhi } 
( \, (x_1, \ldots , x_n) \mid V )
\prod_{  \begin{array}{c}
{\scriptstyle W \in \pi}  \\
{\scriptstyle W \, \text{outer}}
\end{array} } \lambda_{ |W| } 
( \, (x_1, \ldots , x_n) \mid W )
\end{equation}
\[
= \sum_{\rho \ll \pi}  (-1)^{|\rho|-|\pi|} 
\prod_{  \begin{array}{c}
{\scriptstyle V \in \rho}  \\
{\scriptstyle V \, \text{inner}}
\end{array} } \beta_{ |V|; \uPhi } ( \, (x_1, \ldots , x_n) \mid V )
\prod_{  \begin{array}{c}
{\scriptstyle W \in \rho}  \\
{\scriptstyle W \, \text{outer}}
\end{array} } \beta_{ |W|; \uChi } ( \, (x_1, \ldots , x_n) \mid W ).
\]

We now let $\pi$ run in $NC(n)$.  Returning to the equality to be 
proved, Equation (\ref{eqn:43b}), we see that its right-hand side 
(obtained by summing over $\pi$ in (\ref{eqn:43d})) is equal to
\begin{equation}    \label{eqn:43e}
\sum_{\pi \in NC(n)}  \sum_{\rho \ll \pi}  (-1)^{|\rho|-|\pi|} 
\prod_{  \begin{array}{c}
{\scriptstyle V \in \rho}  \\
{\scriptstyle V \, \text{inner}}
\end{array} } \beta_{ |V|; \uPhi } ( \, (x_1, \ldots , x_n) \mid V )
\cdot \prod_{  \begin{array}{c}
{\scriptstyle W \in \rho}  \\
{\scriptstyle W \, \text{outer}}
\end{array} } \beta_{ |W|; \uChi } ( \, (x_1, \ldots , x_n) \mid W ).
\end{equation}
We must verify that the quantity in
(\ref{eqn:43e}) is equal to $\chi_n (x_1, \ldots , x_n)$.  To this 
end, we exchange the order of summation over $\pi$ and $\rho$, so 
that (\ref{eqn:43e}) becomes:
\[
\sum_{\rho \in NC(n)}
\bigg( \sum_{   \begin{array}{c}
{\scriptstyle \pi \in NC(n), }  \\
{\scriptstyle \pi \gg \rho}
\end{array} }  (-1)^{|\rho|-|\pi|} \bigg)
\prod_{  \begin{array}{c}
{\scriptstyle V \in \rho}  \\
{\scriptstyle V \, \text{inner}}
\end{array} } 
\beta_{ |V|; \uPhi } \bigl( \, (x_1, \ldots , x_n) \mid V \bigr)
\times
\]
\begin{equation}   \label{eqn:43f}
\mbox{$\ $} \hspace{3cm}
\times  \prod_{  \begin{array}{c}
{\scriptstyle W \in \rho}  \\
{\scriptstyle W \, \text{outer}}
\end{array} } 
\beta_{ |W|; \uChi } \bigl( \, (x_1, \ldots , x_n) \mid W \bigr).
\end{equation}
It was however noticed in Section 2 that for a fixed $\rho \in NC(n)$ 
one has:
\[
\sum_{   \begin{array}{c}
{\scriptstyle \pi \in NC(n), }  \\
{\scriptstyle \pi \gg \rho}
\end{array} }  (-1)^{|\rho|-|\pi|} 
= \left\{    \begin{array}{ll}
1 &  \text{ if } \rho \in \Int(n),\\
0 &  \text{ otherwise. }
\end{array}  \right.
\]
Therefore the summation in (\ref{eqn:43f}) reduces to just
$\sum_{\rho \in \Int(n)}  \ \prod_{W \in \rho}
\beta_{ |W| ; \uChi } 
\bigl( \, (x_1, \ldots , x_n) \mid W \bigr)$
(where we also took into account that all blocks of an interval 
partition are outer blocks).  The latter sum is indeed equal to
$\chi_n (x_1, \ldots , x_n)$, as required, by the definition 
of Boolean cumulants.
\end{proof}

$\ $

\begin{lemma} \label{lem:44}
Let $n \geq 2$ be an integer, let $V_o$ be a subset of 
$\{ 1, \ldots , n \}$ such that $V_o \ni 1,n$, and let 
$x_1, \ldots , x_n$ be in $\cV$.  One has:
\begin{equation}  \label{eqn:44a}
\sum_{ \begin{array}{c}
	{\scriptstyle \pi \in NC(n) \ such}  \\
	{\scriptstyle that \ V_o \in \pi}
	\end{array} } \ (-1)^{1 + | \pi |} 
\ \prod_{  \begin{array}{c}
           {\scriptstyle V \in \pi,}   \\
           {\scriptstyle V \neq V_o}
           \end{array}  }
\ \beta_{ |V|; \uPhi }  \bigl( \, (x_1, \ldots , x_n) \mid V \bigr)
\end{equation}
\[
= \sum_{ \begin{array}{c}
	{\scriptstyle \rho \in NC(n) \ such}  \\
	{\scriptstyle that \ V_o \in \rho}
	\end{array} } \ \Moeb_n ( \rho , 1_n)
\cdot  \prod_{  \begin{array}{c}
	{\scriptstyle V \in \rho}  \\
	{\scriptstyle V \neq V_o}
	\end{array} } 
\ \varphi_{ |V| } \bigl( \, (x_1, \ldots , x_n) \mid V \bigr).
\]   
\end{lemma}

$\ $

\begin{proof}
Fix for the moment a partition $\pi \in NC(n)$ such that $V_o \in \pi$.
For every block $V \neq V_o$ of $\pi$, we use the definition of 
Boolean cumulants in order to express 
$\beta_{ |V|; \uPhi } ( \, (x_1, \ldots , x_n) \mid V )$ as a sum 
indexed by $\Int ( |V| )$, and then we cross-multiply the resulting 
sums.  The reader should have no difficulty to verify that, upon 
doing this cross-multiplication, one arrives to the formula
\begin{equation}   \label{eqn:44b}
\prod_{  \begin{array}{c}
{\scriptstyle V \in \pi,}  \\
{\scriptstyle V \neq V_o} \end{array} } \, \beta_{ |V| }  
\bigl( \, (x_1, \ldots , x_n) \mid V \bigr) 
= \sum_{ \begin{array}{c}
{\scriptstyle \rho \sqsubseteq \pi \ such}  \\
{\scriptstyle that \ V_o \in \rho}
\end{array} } \, (-1)^{|\rho | - | \pi |} 
\ \prod_{ \begin{array}{c}
{\scriptstyle V \in \rho,}   \\  
{\scriptstyle V \neq V_o}
\end{array} }
\, \varphi \bigl( \, (x_1, \ldots , x_n) \mid V \bigr),
\end{equation}
where $\sqsubseteq$ is the partial order relation on $NC(n)$ 
reviewed in Section 2.2.

We now let $\pi$ run in the index set shown on the left-hand side 
of (\ref{eqn:44a}).  By summing in Equation (\ref{eqn:44b}) over this 
range for $\pi$, we find that
$$\sum_{ \begin{array}{c}
		{\scriptstyle \pi \in NC(n) \ such}  \\
		{\scriptstyle that \ V_o \in \pi}
		\end{array} } \ (-1)^{ | \pi | - 1} \ \prod_{V \in \pi, \, V \neq V_o}
	\ \beta_{ |V| }  \bigl( \, (a_1, \ldots , a_n) \mid V \bigr)$$
$$= \sum_{ \begin{array}{c}
		{\scriptstyle \pi \in NC(n) \ such}  \\
		{\scriptstyle that \ V_o \in \pi}
		\end{array} } \sum_{ \begin{array}{c}
		{\scriptstyle \rho \sqsubseteq \pi \ such}  \\
		{\scriptstyle that \ V_o \in \rho}
		\end{array} } \ (-1)^{|\rho | - 1} \ \prod_{V \in \rho, \, V \neq V_o}
	\ \varphi \bigl( \, (a_1, \ldots , a_n) \mid V \bigr).$$
By exchanging sums indexed by $\rho$ and $\pi$, we can continue the above with
\[
= \sum_{ \begin{array}{c}
{\scriptstyle \rho \in NC(n) \ such}  \\
{\scriptstyle that \ V_o \in \rho}
\end{array} }  \ (-1)^{|\rho | - 1} |\{\pi \in NC(n) : \rho \sqsubseteq \pi \}|
\cdot  \prod_{  \begin{array}{c}
{\scriptstyle V \in \rho}  \\
{\scriptstyle V \neq V_o}
\end{array} } 
\ \varphi_{ |V| } \bigl( \, (a_1, \ldots , a_n) \mid V \bigr).
\]

In order to conclude the proof, we are left to verify that for every 
$\rho \in NC(n)$ such that $V_o \in \rho$, one has
\[
(-1)^{|\rho | - 1} 
\left| \{\pi \in NC(n) : \rho \sqsubseteq \pi \}\right| 
= \Moeb_n ( \rho , 1_n).
\]
To this end, we invoke Lemma \ref{lem:210}, which gives us that 
\[
\left| \{\pi \in NC(n) : \rho \sqsubseteq \pi \}\right| 
= \left| \{\pi\in NC(n): \pi\ll \rho\}\right| .
\]
It was noticed in Section 2 that the latter cardinality is equal to 
the product of Catalan numbers $\prod_{V\in\rho} C_{|V|-1}$, and upon 
multiplying this with $(-1)^{|\rho | - 1}$, one arrives indeed to the 
required value $\Moeb_n ( \rho , 1_n)$.
\end{proof}

$\ $

\noindent
{\bf Proof of Proposition \ref{prop:41}.} 
The case when $n=1$ is clear (both sides of Equation 
(\ref{eqn:41a}) are equal to $\chi_1 (x_1)$).  For $n \geq 2$,
we start from the expression for 
$\kappa_n^{(c)} (x_1, \ldots , x_n)$
found in Lemma \ref{lem:43}, where we write the sum on the 
right-hand side of Equation (\ref{eqn:43a}) as a double sum
of the form
\begin{equation}   \label{eqn:41b}
\sum_{ \begin{array}{c}  
      {\scriptstyle V_o \subseteq \{ 1, \ldots , n \}}  \\
      {\scriptstyle with \ V_o \ni 1,n}
      \end{array} }
\ \sum_{ \begin{array}{c}  
      {\scriptstyle \pi \in NC(n) \ such } \\
      {\scriptstyle that \ V_o \in \pi}
      \end{array} } 
\ \mbox{[term indexed by $V_o$ and $\pi$].}
\end{equation}
Then for every $V_o \subseteq \{ 1, \ldots , n \}$ such that 
$V_o \ni 1,n$ we replace the second sum in (\ref{eqn:41b}) by 
using Lemma \ref{lem:44}.  When doing this replacement we 
arrive to a double sum of the form
\[
\sum_{ \begin{array}{c}  
      {\scriptstyle V_o \subseteq \{ 1, \ldots , n \}}  \\
      {\scriptstyle with \ V_o \ni 1,n}
      \end{array} }
\ \sum_{ \begin{array}{c}  
      {\scriptstyle \rho \in NC(n) \ such } \\
      {\scriptstyle that \ V_o \in \rho}
      \end{array} } 
\ \mbox{[term indexed by $V_o$ and $\rho$],}
\]
and converting the latter double sum into one sum indexed 
by $\{ \rho \in NC(n) : \rho \ll 1_n \}$ leads to the 
required formula (\ref{eqn:41b}).
\hfill  $\square$

$\ $

\section{Non-crossing partitions with symmetries,  \\
and a variation on c-free cumulant functionals}  

The ``typical'' combinatorial approach to any brand of cumulants 
goes by identifying a coherent sequence of {\em lattices of partitions},
which are used as index sets for the various summation formulas 
describing the cumulants in question.  This applies in particular to 
the free cumulants and to the Boolean cumulants reviewed in Sections 
3.1 and 3.2, where the relevant lattices of partitions are 
$\{ NC(n) : n \in \bN \}$ and respectively
$\{ \Int (n) : n \in \bN \}$.  The approach with partitions is 
good because it offers streamlined ideas on how to proceed -- for 
instance one typically uses (following a method initiated by Rota 
in the 1960's) a notion of ``multiplicative functions'' on the relevant 
sequence of lattices, then one looks for a formula for 
``cumulants with products as entries'', where an essential ingredient 
in the formula is the use of lattice operations. 

In this section we review the approach to cumulants via partitions 
in the two extended frameworks we are considering: c-free cumulants 
and infinitesimal free cumulants, and we point out a relevant fact 
which seems to have been overlooked in the c-free case.

\subsection{Lattices of partitions for infinitesimal free cumulants.}

$\ $

\noindent
We first go over the (better understood) case of infinitesimal 
free cumulants.  In order to approach these cumulants via the study 
of multiplicative functions on a sequence of lattices, one uses an 
``analogue of type B'' for the $NC(n)$'s; this is a family of lattices 
denoted as $\NCB (n)$ which were introduced in \cite{Re1997}.
In order to depict a partition $\sigma \in \NCB (n)$, one starts by 
marking on a circle $2n$ points, labelled as $1, \ldots , n$ and 
$-1 , \ldots , -n$ (as exemplified, for $n=5$, in Figure 2 below).  
Then $\sigma$ must achieve a non-crossing partition of these $2n$ 
points, with the additional symmetry requirement that if $U$ is 
a block of $\sigma$ then $-U := \{ -i : i \in U \}$ is a block 
of $\sigma$ as well.

Note that a block $U$ of a $\sigma \in NC^{(B)} (n)$ either is 
such that $U = -U$, in which case we say that $U$ is a 
{\em zero-block} of $\sigma$, or is such that 
$U \cap (-U) = \emptyset$.  The non-crossing condition forces 
every $\sigma \in \NCB (n)$ to have at most one zero-block; 
moreover, there exists a natural bijection between
$\{ \sigma \in \NCB (n) : \sigma \mbox{ has a zero-block} \}$
and
$\{ \sigma \in \NCB (n) : \sigma \mbox{ has no zero-block} \}$,
which is implemented by a suitable version of the Kreweras 
complementation map.

In the papers \cite{BiGoNi2003} and then \cite{FeNi2010} 
it was pointed out that, when used on the lattices 
$\NCB (n)$, the general machinery of Rota 
leads to the infinitesimal free cumulant functionals
$\kappa_n ' : \cA^n \to \bC$ associated to an I-ncps
$( \cA , \varphi , \varphi ' )$.  The same considerations apply, 
without any changes, to the framework of 
$( \cV, \uPhi , \uPhi ' )$ used in Section 3.3.  Since the 
families of equations (\ref{eqn:34a}) and (\ref{eqn:34b}) shown 
in Section 3.3 did not, at the face of it, call on the lattices
$\NCB (n)$, it is of relevance to mention here that 
the double sum ``$\sum_{\pi \in NC(n)} \sum_{V \in \pi}$'' 
appearing in these families of equations comes from processing
a single sum indexed by $\{ \sigma \in \NCB (n) : \sigma 
\mbox{ has a zero-block} \}$.  For instance, Equation 
(\ref{eqn:34a}) corresponds to a formula which looks like this:
\begin{equation}   \label{eqn:5a}
\varphi_n ' (x_1, \ldots , x_n )
= \sum_{ \begin{array}{c}
{\scriptstyle \sigma \in \NCB (n)}   \\
{\scriptstyle with \ zero-block \ Z}
\end{array} } 
\kappa_{|Z|/2} ' \bigl( (x_1, \ldots , x_n) \mid \Abs (Z) \bigr)
\times
\end{equation}
\[
\mbox{$\ $} \hspace{3cm} \times 
\prod_{ \begin{array}{c}
{\scriptstyle pairs \ U, -U \in \sigma} \\
{\scriptstyle such \ that \ U \neq (-U) }
\end{array}  }  \, 
\kappa_{|U|; \uPhi} \bigl( (x_1, \ldots , x_n) \mid \Abs (U) \bigr),
\]
where $\Abs : \{ 1, \ldots , n \} \cup \{ -1, \ldots , -n \} 
\to \{ 1, \ldots , n \}$ is the absolute value map sending 
$\pm i$ to $i$ for $1 \leq i \leq n$.  (For the details of how one 
passes back-and-forth between Equations (\ref{eqn:34a}) and 
(\ref{eqn:5a}) we refer the reader to Section 6 of \cite{FeNi2010}.)

$\ $

\begin{center}
		\begin{tikzpicture}[scale=2.5,cap=round,>=latex]
		
		\draw[thick] (0cm,0cm) circle(1cm);
		
		\foreach \x in {18,54,...,342} {
			\filldraw[black] (\x:1cm) circle(0.4pt);
		}
		\draw[black] (162:1cm)--(90:1cm);
		\draw[black] (90:1cm)--(342:1cm);
		\draw[black] (342:1cm)--(270:1cm);	
		\draw[black] (270:1cm)--(162:1cm);
		
		\draw[black] (18:1cm) to[out=180,in=270] (54:1cm);	
		\draw[black] (234:1cm) to[out=90,in=0] (198:1cm);	
		
		
		\foreach \x/\xtext in {
			198/-5,
			234/-4,
			270/-3,
			306/-2,
			342/-1,
			18/5,
			54/4,
			90/3,
			126/2,
			162/1
		}
		\draw (\x:1.25cm) node[] {$\xtext$};
		
		\end{tikzpicture}
		
{\bf Figure 2.} 
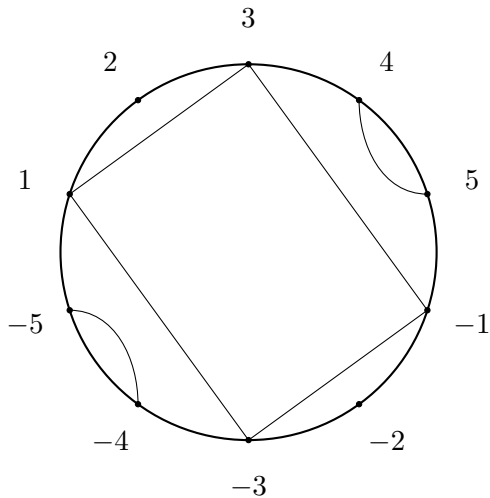
{\em The partition 
$\{ \, \{ 1,3,-1,-3 \}, \{ 2 \}, \{ -2 \}, \{ 4,5 \} , 
\{-4,-5\} \, \} \in \NCB (5)$.}
\end{center}

$\ $

\subsection{Lattices of partitions for conditionally free cumulants.}

$\ $

\noindent
Now let us look at the $c$-free framework.  In order to approach 
cumulants via the study of multiplicative functions on a 
sequence of lattices, one uses here some lattices $\NCBopp (n)$ 
which were identified in \cite{Ca2000}.  The drawing of a 
$\sigma \in \NCBopp (n)$ starts by marking on a circle $2n$ 
points, labelled as $1, \ldots , n$ and $-n , \ldots , -1$ (as 
exemplified, for $n=5$, in Figure 3 below).  Then $\sigma$ must 
achieve a non-crossing partition of these $2n$ 
points, with the additional symmetry requirement that if $U$ is 
a block of $\sigma$ then $-U := \{ -i : i \in U \}$ is a block 
of $\sigma$ as well.  Note that the description of how we obtain 
our $\sigma$ is strikingly similar to the description which 
immediately preceded Figure 2 -- the only difference (and the 
reason for using the name ``$\NCBopp (n)$'') is that the points 
with negative labels $-1, \ldots , -n$ now appear 
{\em in reverse order} as we travel around the circle.

Similar to how things went for partitions in $NC^{(B)}(n)$, a block 
$U$ of a partition $\sigma$ in $NC^{(B-opp)} (n)$ either is such 
that $U = -U$, in which case we say that $U$ is a {\em zero-block} 
of $\sigma$, or is such that $U \cap (-U) = \emptyset$.  But
unlike how things went for $NC^{(B)} (n)$, a partition 
$\sigma \in NC^{(B-opp)} (n)$ may have multiple zero-blocks (for 
instance the second partition shown in Figure 3 has two such blocks).

In the paper \cite{Ca2000} it was shown that, when used
on the lattices $\NCBopp (n)$, the general machinery of Rota 
leads to the identification of a family of cumulants 
$( \kappa^{(cc)}_n : \cA^n \to {\mathbb C} )_{n=1}^{\infty}$ 
associated to a C-ncps $( \cA , \varphi, \chi )$, which can be 
used for describing $c$-freeness for subalgebras of $\cA$.  
A puzzling detail appearing at this point is that the 
functionals $\kappa^{(cc)}_n$ are not the same as the
$\kappa^{(c)}_n$'s coming from \cite{BoLeSp1996}, which were 
reviewed in Section 3.4.  The main goal of the present section 
is to resolve this puzzle by pointing out a neat direct 
connection between the $\kappa_n^{(c)}$'s and 
$\kappa_n^{(cc)}$'s, in Proposition \ref{prop:54} below.
We start towards that by recording a few elementary observations
about the structure of the partitions in $NC^{(B-opp)} (n)$.

\vspace{10pt}

\begin{center}
	\begin{tikzpicture}[scale=2.5,cap=round,>=latex]
	
	\draw[thick] (0cm,0cm) circle(1cm);

	\foreach \x in {18,54,...,342} {
		\filldraw[black] (\x:1cm) circle(0.4pt);
	}
	\draw[black] (162:1cm)--(90:1cm);
	\draw[black] (90:1cm)--(270:1cm);
	\draw[black] (270:1cm)--(198:1cm);	
	\draw[black] (198:1cm)to[out=0,in=0](162:1cm);
	
	\draw[black] (18:1cm) to[out=180,in=270] (54:1cm);	
	\draw[black] (306:1cm) to[out=90,in=180] (342:1cm);	
	
	
	\foreach \x/\xtext in {
		198/-1,
		234/-2,
		270/-3,
		306/-4,
		342/-5,
		18/5,
		54/4,
		90/3,
		126/2,
		162/1
	}
	\draw (\x:1.25cm) node[] {$\xtext$};
	
	\begin{scope}[shift={(3,0)}]
		\draw[thick] (0cm,0cm) circle(1cm);

		\foreach \x in {18,54,...,342} {
			\filldraw[black] (\x:1cm) circle(0.4pt);
		}
		\draw[black] (162:1cm)--(90:1cm);
		\draw[black] (90:1cm)--(270:1cm);
		\draw[black] (270:1cm)--(198:1cm);	
		\draw[black] (198:1cm) to[out=0,in=0] (162:1cm);

	\draw[black] (18:1cm) to[out=180,in=270] (54:1cm);	
	\draw[black] (18:1cm) to[out=180,in=180] (342:1cm);	
	\draw[black] (306:1cm) to[out=90,in=180] (342:1cm);		
	\draw[black] (306:1cm)--(54:1cm);	
		
		\foreach \x/\xtext in {
			198/-1,
			234/-2,
			270/-3,
			306/-4,
			342/-5,
			18/5,
			54/4,
			90/3,
			126/2,
			162/1
		}
		\draw (\x:1.25cm) node[] {$\xtext$};
		
	\end{scope}
	
	\end{tikzpicture}	
	
{\bf Figure 3.} 
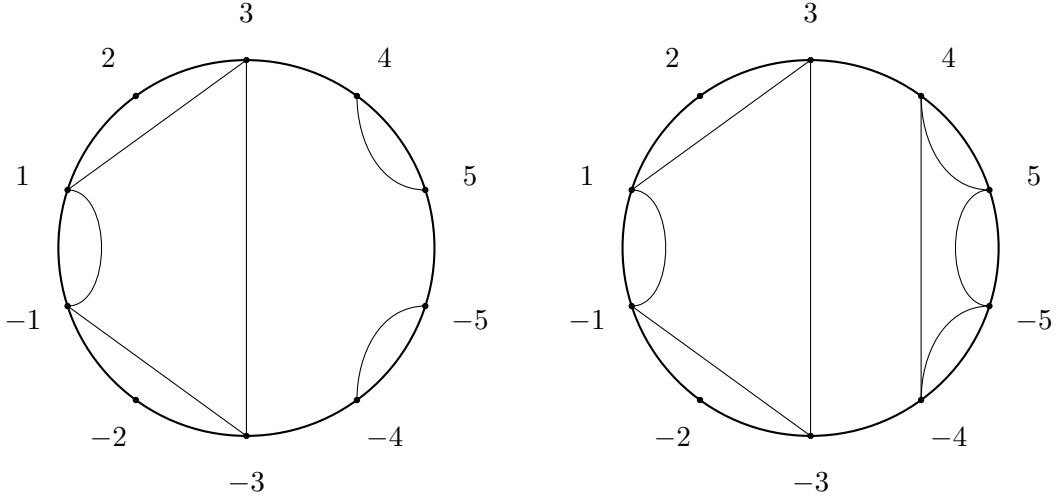
{\em The partitions 
$\{ \, \{ 1,3,-1,-3 \}, \{ 2 \}, \{ -2 \}, \{ 4,5 \} , 
\{-4,-5\} \, \}$ (left)}

{\em and 
$\{ \, \{ 1,3,-1,-3 \}, \{ 2 \}, \{ -2 \}, 
\{ 4,5, -4,-5\} \, \}$ (right) in $\NCBopp (5)$.} 
\end{center}

\vspace{10pt}

\begin{remark}  \label{rem:51}
Let $\sigma$ be a partition in $NC^{(B-opp)} (n)$.

\vspace{6pt}

(1) If $U$ is a block of $\sigma$ such that 
$U \cap \{ 1, \ldots , n \} \neq \emptyset \neq
U \cap \{ -1, \ldots , -n \}$,
then $U$ must be a zero-block.  Indeed, let 
$i,j \in \{ 1, \ldots , n \}$ be such that $i, -j \in U$.  If 
$i=j$, then $i \in U \cap (-U)$, which forces $U = -U$.  If 
$i \neq j$, then we look at the four points $i,j,-i,-j$ drawn 
around the circle and we observe that if $U$ and $-U$ would 
be distinct blocks of $\sigma$, then there would be crossings 
between them (e.g. if $i<j$, then the four points we're 
looking at come in the order $i,j,-j,-i$, with $i,-j \in U$ 
and $j,-i \in -U$).  So this case, too, leads to the conclusion 
that $U = -U$.

\vspace{6pt}

(2) The contrapositive of (1) is that every non-zero-block of 
$\sigma$ either is contained in $\{ 1, \ldots , n \}$ or is 
contained in $\{ -1, \ldots , -n \}$.  Thus the non-zero blocks 
of $\sigma$ come in pairs $U, -U$, with one of $U, -U$ 
contained in $\{ 1, \ldots , n \}$ and the other contained 
in $\{ -1, \ldots , -n \}$.
\end{remark}

\begin{defrem}  \label{def:52}
Let $\cV$ be a vector space over $\bC$, and consider (analogous 
to the considerations of Section 3.4) a triple $( \cV , \uPhi, \uChi )$ 
where $\uPhi = ( \varphi_n )_{n=1}^{\infty}$ and 
$\uChi = ( \chi_n )_{n=1}^{\infty}$ are families of multilinear
functionals on $\cV$.  We will use the name of 
{\em alternative c-free cumulants} associated to 
$( \cV , \uPhi, \uChi )$ for the family of multilinear functionals 
$\uKappa^{(cc)} = ( \kappa^{(cc)}_n )_{n=1}^{\infty}$ 
which is uniquely determined by the requirement that
\begin{equation}   \label{eqn:52a}
\chi_n (x_1 , \ldots , x_n) = 
\end{equation}
\[
\sum_{\sigma \in NC^{(B-opp)} (n)} 
\prod_{  \begin{array}{c}
        {\scriptstyle U \in \sigma,}  \\
        {\scriptstyle U \subseteq \{ 1, \ldots , n \} }
\end{array} } \kappa_{ |U| ; \uPhi } 
( \, (x_1, \ldots , x_n) \mid U )
\prod_{  \begin{array}{c}
        {\scriptstyle Z \in \sigma,}  \\
        {\scriptstyle Z = -Z}
\end{array} } \kappa^{(cc)}_{ |Z|/2 } 
( \, (x_1, \ldots , x_n) \mid Z \cap \{ 1, \ldots , n \} ),
\]
holding for all $n \in \bN$ and $x_1, \ldots , x_n \in \cV$,
and where $( \kappa_{n ; \uPhi} )_{n=1}^{\infty}$ 
are the free cumulants associated to $( \cV , \uPhi )$.

The fact that $\uKappa^{(cc)}$ can indeed be defined 
by using Equation (\ref{eqn:52a}) is easily seen, in a similar
way to how it was seen that the c-free cumulants $\uKappa^{(c)}$
associated to $( \cV , \uPhi, \uChi )$ are correctly defined by 
Equation (\ref{eqn:33a}) of Section 3.4.  In the case at hand, 
one isolates on the right-hand side of (\ref{eqn:52a}) the term 
$\kappa_n^{(cc)} (x_1, \ldots , x_n)$ which is indexed by the 
partition $\sigma \in NC^{(B-opp)}$ with only one block, 
and then one proceeds by induction.

Before stating the proposition which relates $\uKappa^{(cc)}$ 
to $\uKappa^{(c)}$ we record another remark about $\NCBopp$, 
concerning the natural ``absolute value map''
$\Abs : \NCBopp (n) \to NC(n)$.
\end{defrem}

\begin{remnotation}    \label{rem:53}
Let $n$ be a positive integer, and for every subset 
$U \subseteq \{ 1, \ldots ,n \} \cup \{ -1, \ldots , -n \}$,
let us agree to denote $\Abs (U) := \{ |i| : i \in U \}$.

\vspace{6pt}

(1) It is immediate that if $\sigma$ is a partition in 
$\NCBopp (n)$ and if $U_1, U_2$ are two blocks of $\sigma$,
then either $\Abs (U_1) = \Abs (U_2)$ or
$\Abs (U_1) \cap \Abs (U_2) = \emptyset$. This implies that
the set of sets 
\[
\Abs ( \sigma ) := \{ \Abs (U) : U \in \sigma \}
\]
is a partition of $\{ 1, \ldots , n \}$.  It is, moreover, 
easily seen that $\Abs ( \sigma )$ must belong to $NC(n)$.
Indeed, for any two distinct blocks $V_1, V_2$ of 
$\Abs ( \sigma )$ it is possible to pick two distinct blocks 
$U_1, U_2$ of $\sigma$ such that $V_1 \subseteq U_1$ and 
$V_2 \subseteq U_2$; so a crossing between $V_1$ and $V_2$
would entail a crossing between $U_1$ and $U_2$, which is 
not possible. 

\vspace{6pt}

(2) Let $\sigma$ be in $\NCBopp (n)$, and consider the 
partition $\pi = \Abs ( \sigma ) \in NC(n)$.  We note that
if $Z$ is a zero-block of $\sigma$, then $W := \Abs (Z)$ 
has to be an outer block of $\pi$.  Indeed, if $W$ was to 
be nested inside some other block $V$ of $\pi$, then upon
writing $V = \Abs (U)$ for some $U \in \sigma$ one would 
immediately find crossings between $Z$ and $U$, which is 
not possible.

\vspace{6pt}

(3) Parts (1) and (2) of this remark provide us with a map
\begin{equation}   \label{eqn:53a}
\sigma \mapsto 
\Bigl( \,  \Abs ( \sigma ), 
\,  \{ \Abs (Z) : Z \in \sigma, \ Z = -Z \} \, \Bigr)
\end{equation}
going from $\NCBopp (n)$ to the set
$\{ ( \pi , \cS )  : \pi \in NC(n), \, \cS \subseteq 
\Out ( \pi ) \}$, where we used the notation 
$\Out ( \pi )$ for the set of all outer blocks of a 
partition $\pi \in NC(n)$.  We leave it as an exercise to the 
reader to check that the map (\ref{eqn:53a}) is a bijection, 
with inverse described as follows: given $\pi \in NC(n)$ and 
a set $\cS$ (possibly empty) of outer blocks of $\pi$, we put 
\[
\sigma := \{ W \cup (-W) :  W \in \cS \}
\cup \{ V : V \in \pi \setminus \cS \}
\cup \{ -V : V \in \pi \setminus \cS \}
\]
(where $\pi \setminus \cS$ denotes the set of blocks of $\pi$
not taken in $\cS$).
\end{remnotation}

\begin{proposition}    \label{prop:54}
Let $\cV$ be a vector space over $\bC$, and consider 
a triple $( \cV , \uPhi, \uChi )$ where 
$\uPhi = ( \varphi_n )_{n=1}^{\infty}$ and 
$\uChi = ( \chi_n )_{n=1}^{\infty}$ are in $\fMcV$.
Let $\uKappa^{(c)} = ( \kappa^{(c)}_n )_{n=1}^{\infty}$ 
be the c-free cumulants associated to $( \cV , \uPhi, \uChi )$
as in Section 3.4, 
and let $\uKappa^{(cc)} = ( \kappa^{(cc)}_n )_{n=1}^{\infty}$ 
be the alternative c-free cumulants considered in 
Definition \ref{def:52}.  One has
\begin{equation}   \label{eqn:54a}
\kappa^{(cc)}_n  =  \kappa^{(c)}_n - \kappa_{n; \uPhi},
\ \ n \in \bN,
\end{equation}	
where $( \kappa_{n ; \uPhi} )_{n=1}^{\infty}$ 
are the free cumulants associated to $( \cV , \uPhi )$.
\end{proposition}

\begin{proof}
The bijection observed in Remark \ref{rem:53}(3) can be used as a 
change of variable in the Equation (\ref{eqn:52a}) defining 
$\uKappa^{(cc)}$, which then takes the form:
\begin{equation}   \label{eqn:54b}
\chi_n (x_1 , \ldots , x_n) = 
\end{equation}
\[
\sum_{\pi \in NC(n)}
\sum_{\cS \subseteq \Out ( \pi )}
\prod_{V \in \pi \setminus \cS} 
\kappa_{ |V| ; \uPhi } ( \, (x_1, \ldots , x_n) \mid V )
\cdot \prod_{W \in \cS}
\kappa^{(cc)}_{ |W| } ( \, (x_1, \ldots , x_n) \mid W ),
\]
holding for all $n \in \bN$ and $x_1, \ldots , x_n \in \cV$.
It is immediate that if on the right-hand side of (\ref{eqn:54b}) 
we fix a $\pi \in NC(n)$ and we only perform the sum over 
$\cS \subseteq \Out ( \pi )$, what results is the product
\begin{equation}   \label{eqn:54c}
\prod_{  \begin{array}{c}
        {\scriptstyle V \in \pi}  \\
        {\scriptstyle V \, \text{inner}}
         \end{array} } 
\kappa_{ |V| ; \uPhi } ( (x_1, \ldots , x_n) \mid V )
\cdot \prod_{  \begin{array}{c}
              {\scriptstyle W \in \pi}  \\
              {\scriptstyle W \, \text{outer}}
               \end{array} } 
\bigl( 
\kappa_{ |W| ; \uPhi } ( (x_1, \ldots , x_n) \mid W )
+ \kappa^{(cc)}_{ |W| } ( (x_1, \ldots , x_n) \mid W ) 
\bigr).
\end{equation}
 
For every $n \in \bN$ let us put 
$\lambda_n := \kappa_n^{(cc)} + \kappa_{n; \uPhi}$,
and let $\uLambda := ( \lambda_n )_{n=1}^{\infty} \in \fMcV$.
Upon replacing (\ref{eqn:54c}) inside
the right-hand side of Equation (\ref{eqn:54b}), we find that 
\begin{equation}   \label{eqn:54d}
\chi_n (x_1 , \ldots , x_n) = 
\end{equation}
\[
\sum_{\pi \in NC(n)}
\prod_{  \begin{array}{c}
        {\scriptstyle V \in \pi}  \\
        {\scriptstyle V \, \text{inner}}
         \end{array} } 
\kappa_{ |V| ; \uPhi } ( (x_1, \ldots , x_n) \mid V )
\cdot \prod_{  \begin{array}{c}
              {\scriptstyle W \in \pi}  \\
              {\scriptstyle W \, \text{outer}}
               \end{array} } 
\lambda_{ |W| } ( \, (x_1, \ldots , x_n) \mid W ),
\]
holding for every $n \in \bN$ and $x_1, \ldots , x_n \in \cV$.  We have 
thus obtained that $\uLambda$ satisfies the family of equations 
(\ref{eqn:33a}) which were used to define $\uKappa^{(c)}$.  It follows 
that $\uLambda = \uKappa^{(c)}$, which concludes the proof. 
\end{proof}

\begin{remark}   \label{rem:55}
(1) Let $( \cA, \varphi , \chi )$ be a C-ncps, and let 
$\cA_1, \ldots , \cA_k$ be unital subalgebras of $\cA$.
Proposition \ref{prop:54} explains in a rather neat way why the
functionals $\kappa^{(cc)}_n$ can indeed be used in the description 
of c-free independence, in a similar way to how the $\kappa^{(c)}_n$ 
are used.  Indeed, in the presence of the background condition that 
$\cA_1, \ldots , \cA_k$ are freely independent with respect to $\varphi$
(which imposes the vanishing of the mixed free cumulants with respect 
to $\varphi$), Proposition \ref{prop:54} assures us that the vanishing 
of mixed cumulants $\kappa_n^{(cc)}$ is equivalent to the one of mixed 
cumulants $\kappa_n^{(c)}$.

\vspace{6pt}

(2) $\uKappa^{(cc)}$ is useful because it can be related to the 
lattice operations on $\NCBopp (n)$.  This allows, for instance, a
nicely streamlined treatment of the formula for cumulants with 
products as entries (as shown in Theorem 2.6 of \cite{Ca2000}).

\vspace{6pt}

(3) The formula (\ref{eqn:52a}) which was used for introducing 
$\uKappa^{(cc)}$ in Definition \ref{def:52} can be re-written in 
a way which only uses partitions with zero-blocks on the right-hand 
side.  Indeed, it is immediate that the sub-sum over 
$\{ \sigma \in \NCBopp (n) : \sigma$ has no zero-blocks$\}$ 
on the right-hand side of (\ref{eqn:52a}) simply gives 
$\varphi_n (x_1, \ldots , x_n)$.  With this observation,
Equation (\ref{eqn:52a}) can be put in the form
\begin{equation}   \label{eqn:55a}
( \chi_n -\varphi_n ) (x_1 , \ldots , x_n) = 
\sum_{ \begin{array}{c}
{\scriptstyle \sigma \in \NCBopp (n)}  \\
{\scriptstyle with \ zero-blocks}
\end{array} }
\prod_{  \begin{array}{c}
        {\scriptstyle Z \in \sigma,}  \\
        {\scriptstyle Z = -Z}
\end{array} } \kappa^{(cc)}_{ |Z|/2 } 
( \, (x_1, \ldots , x_n) \mid \Abs (Z)  )  \times
\end{equation}
\[
\mbox{$\ $} \hspace{3cm} \times
\prod_{  \begin{array}{c}
        {\scriptstyle pairs \, U, -U \in \sigma,}  \\
        {\scriptstyle such \ that \, U \neq -U}
\end{array} } \kappa_{ |U| ; \uPhi } 
( \, (x_1, \ldots , x_n) \mid \Abs (U) );
\]
this is strikingly similar to the formula (\ref{eqn:5a}) 
which was reviewed in Section 5.1 in the framework of 
infinitesimal free cumulants.

The apparent resemblance between Equations (\ref{eqn:5a}) 
and (\ref{eqn:55a}) is, in some sense, a manifestation of 
the resemblance between the two types of pictures shown in 
Figures 2 and 3, for partitions in $\NCB (n)$ and in $\NCBopp (n)$.  
We should point out here the intriguing additional detail that
the lattices $\NCB (n)$ and $\NCBopp (n)$ have the same cardinality 
-- they are both counted by the binomial coefficient $2n$-choose-$n$. 
Tantalizing as this may be, we are not aware of any result that 
would relate c-free independence to infinitesimal free independence 
via a direct connection between $\NCB (n)$ and $\NCBopp (n)$; an 
example of a rather natural bijection $\NCB (n) \to \NCBopp (n)$, 
based on systems of parentheses, was explained to one of us by 
Vic Reiner \cite{Re2009}, but that does not seem to convert well 
into considerations on non-commutative random variables.
\end{remark}

$\ $

\section{Proof of Theorem \ref{thm:17}}  

\begin{notation}   \label{def:61}
{\em (Framework of the section.)}

\noindent
Throughout this section we fix a vector space $\cV$ over $\bC$ 
and two families of multilinear functionals 
$\uPhi , \uChi \in \fMcV$, where
$\uPhi = ( \varphi_n )_{n=1}^{\infty}$ and
$\uChi = ( \chi_n )_{n=1}^{\infty}$, as in Notation \ref{def:15}.
We assume that $\uPhi$ is tracial, in the sense indicated in that
same notation.  We also fix a linear map 
$\Delta : \cV \to \cV \otimes \cV$, and 
we consider the transformation $\Deltastar : \fMcV \to \fMcV$ 
constructed from $\Delta$ in the way described in Notation 
\ref{def:16}.  We then put 
$\uPhi ' := \Deltastar \left( \uBeta_{\uChi} \right)$, where 
$\uBeta_{\uChi} = ( \beta_{n; \uChi} )_{n=1}^{\infty} \in \fMcV$
is the family of Boolean cumulants of $\uChi$.  Our goal for the 
section is to prove the equality claimed in Theorem \ref{thm:17}:
\[
\uKappa ' = \Deltastar ( \, \uKappa^{(c)} \, ),
\]
where $\uKappa ' = ( \kappa_n ')_{n=1}^{\infty} \in \fMcV$ are 
the infinitesimal free cumulants associated to 
$( \cV , \uPhi , \uPhi ')$, while
$\uKappa^{(c)} = ( \kappa_n^{(c)} )_{n=1}^{\infty}$ are the 
c-free cumulants associated to $( \cV , \uPhi , \uChi )$.
In view of how $\Deltastar$ is defined, proving the latter 
equality amounts to proving that for every $n \in \bN$ one has
\begin{equation}   \label{eqn:61a}
\kappa_n ' =  \sum_{m=1}^n
\kappa_{n+1}^{(c)} \circ \Gamma_{n+1}^m \circ \Delta_n^{(m)},
\end{equation}
with $\Gamma_{n+1} : \cV^{\otimes (n+1)} \to \cV^{\otimes (n+1)}$ 
and $\Delta_n^{(m)} : \cV^{\otimes n} \to \cV^{\otimes (n+1)}$ as
in (1) and (2) of Notation \ref{def:16}.  In Equation (\ref{eqn:61a}), 
$\kappa_n '$ and $\kappa_{n+1}^{(c)}$ are treated as linear 
functionals on $\cV^{\otimes n}$ and respectively $\cV^{\otimes (n+1)}$ 
(rather than multilinear functionals on $\cV^n$ and $\cV^{n+1}$).

In order to handle the left-hand side of Equation (\ref{eqn:61a}), it 
is convenient to introduce the following notation.
\end{notation}

\begin{notation} \label{def:62}
(1) For every $n \in \bN$ and $m \in \{ 1, \ldots , n \}$, we denote
\begin{equation}   \label{eqn:62a}
\gamma_n^{(m)} := \beta_{n+1 ; \uChi} \circ \Gamma_{n+1}^m 
\circ \Delta_n^{(m)}.
\end{equation}
We will view $\gamma_n^{(m)}$, as needed, either as a linear 
functional $\cV^{\otimes n} \to \bC$ or as a multilinear 
functional $\cV^n \to \bC$.

\vspace{6pt}

(2) For every $n \in \bN$, $m \in \{ 1, \ldots , n \}$ and 
$\pi \in NC(n)$ we will denote by $\gamma_{\pi}^{(m)}$ the 
multilinear functional $\cV^n \to \bC$ defined as follows:
\begin{equation}   \label{eqn:62b}
\gamma_{\pi}^{(m)} (x_1, \ldots , x_n) := 
\gamma_{|V_o|}^{(r)} ( ( x_1, \ldots , x_n) \mid V_o ) \cdot
\prod_{  \begin{array}{c}
         {\scriptstyle V \in \pi} \\
         {\scriptstyle V \neq V_o}
         \end{array}  } 
\varphi_{ |V| } ( ( x_1, \ldots , x_n) \mid V ), 
\end{equation}
where $V_o$ denotes the block of $\pi$ which contains the 
number $m$, and $r$ denotes the ``rank of $m$ inside $V_o$''
-- that is, upon writing $V_o = \{ i_1, \ldots , i_p \}$ 
with $i_1 < \cdots < i_p$, one has $m = i_r$.

\vspace{6pt}

\noindent
[Note that part (2) of this notation is an extension of part (1),
since $\gamma_n^{(m)}$ may be retrieved as $\gamma_{\pi}^{(m)}$
for $\pi = 1_n \in NC(n)$.]
\end{notation}

\begin{lemma}   \label{lem:63}
For every $n \in \bN$, one has
\begin{equation}   \label{eqn:63a}
\kappa_n ' = \sum_{\pi \in NC(n)} \sum_{m=1}^n
\Moeb_n ( \pi , 1_n) \, \gamma_{\pi}^{(m)}.
\end{equation}
\end{lemma} 

\begin{proof}  We start from the formula (\ref{eqn:34b}) which defines 
$\kappa_n ' ( x_1, \ldots , x_n)$, and on the right-hand side of that 
formula we replace the quantity
$\varphi_{ | V_o | } ' \bigl( \, (x_1, \ldots , x_n) \mid V_o \bigr)$
by using the definition of $\varphi_{|V_o|} '$, which is
\[
\varphi_{|V_o|} ' = \beta_{|V_o|+1; \uChi} \circ \widetilde{\Delta}_{|V_o|}
= \sum_{j=1}^{|V_o|}
\beta_{|V_o|+1; \uChi} \circ \Gamma_{|V_o|+1; \uChi}^j 
\circ \Delta_{|V_o|}^{(j)}.
\]
It is convenient to re-write the latter sum in the equivalent form
\begin{equation}  \label{eqn:63b}
\sum_{m \in V_o}
\beta_{|V_o|+1; \uChi} \circ \Gamma_{|V_o|+1; \uChi}^{r(m)}
\circ \Delta_{|V_o|}^{(r(m))},
\end{equation} 
where $r(m)$ indicates the rank of $m$ within the block $V_o$ (as indicated
in Notation \ref{def:62}(2)).  Upon substituting (\ref{eqn:63b}) in
(\ref{eqn:34b}), one arrives to an equation of the form
\begin{equation}  \label{eqn:63c}
\kappa_n ' (x_1, \ldots , x_n) 
= \sum_{\pi \in NC(n)} \sum_{V_o \in \pi} \sum_{m \in V_o} 
\Moeb_n ( \pi , 1_n) \cdot \term ( \pi , V_o , m),
\end{equation}
and the reader should have no difficulty to verify that the quantity 
``$\term ( \pi , V_o , m)$'' appearing in (\ref{eqn:63c}) is nothing but 
the $\gamma_{\pi}^{(m)} (x_1, \ldots , x_n)$ from Notation \ref{def:62}(2).  
Finally, the double sum $\sum_{V_o \in \pi} \sum_{m \in V_o}$ in 
(\ref{eqn:63c}) can be re-written as a plain sum $\sum_{m=1}^n$, which leads 
to the formula (\ref{eqn:63a}) stated in the lemma.
\end{proof}

We next introduce a notation which captures the multilinear functionals
used for the description of c-free cumulants in Proposition \ref{prop:41}.

\begin{notation} \label{def:64}
For $n \in \bN$ and $\pi \in NC(n)$ such that $\pi \ll 1_n$ 
we define $\eta_{\pi} : \cV^n \to \bC$ by
\begin{align*}
\eta_{\pi} (x_1,\ldots, x_n)
= \beta_{|W_o| ; \uChi} 
\left( (x_1,\ldots,x_n) | W_o \right)
\cdot
\prod_{ \begin{array}{c} 
        {\scriptstyle V \in \pi,}   \\
        {\scriptstyle V \neq W_o}
        \end{array}  } 
\varphi_{|V|} \left((x_1,\ldots,x_n)|V\right) ,
\end{align*}
where $W_o$ is the block of $\pi$ which contains the numbers
$1$ and $n$.
\end{notation}

\begin{lemma}   \label{lem:65}
For every $n \in \bN$, the right-hand side of Equation (\ref{eqn:61a}) 
can be written as
\begin{equation}   \label{eqn:65a}
\sum_{ \begin{array}{c}
       {\scriptstyle \rho \in NC(n+1),}  \\
       {\scriptstyle \rho \ll 1_{n+1} }  
       \end{array}  }
\sum_{m=1}^n
\Moeb_n ( \rho , 1_{n+1} ) \ \eta_{\rho} \circ
\Gamma_{n+1}^{(m)} \circ \Delta_n^{(m)}.
\end{equation}
\end{lemma} 

\begin{proof} Proposition \ref{prop:41} tells us that 
\[
\kappa_{n+1}^{(c)} = 
\sum_{ \begin{array}{c}
       {\scriptstyle \rho \in NC(n+1),}  \\
       {\scriptstyle \rho \ll 1_{n+1} }  
       \end{array}  }
\Moeb_n ( \rho , 1_{n+1} ) \ \eta_{\rho}.
\]
Substituting this on the right-hand side of Equation (\ref{eqn:61a})
leads to (\ref{eqn:65a}).
\end{proof}

Returning to the formula (\ref{eqn:61a}) that needs to 
be proved: we now have both its sides expressed as sums, 
in Lemmas \ref{lem:63} and \ref{lem:65}.  Our proof
of (\ref{eqn:61a}) will go by showing that the sums appearing 
in the two said lemmas can be identified term by term.  In order 
to do the identification of the indexing sets, we will use the 
bijection described as follows.

\begin{remnotation}   \label{rem:66}
Let $n \in \bN$ and $m \in \{ 1, \ldots , n \}$ be given.
We consider the natural bijections 
\begin{equation}  \label{eqn:66a}
\{ \rho \in NC(n+1) : \rho \ll 1_{n+1} \} \to 
\left\{ \widehat{\rho} \in NC(n+1) \begin{array}{lc}
\vline & \mbox{$m$ and $m+1$ } \\
\vline & \mbox{belong to the}  \\
\vline & \mbox{same block of $\widehat{\rho}$}
\end{array}  \right\} \to NC(n),
\end{equation}  
where:

\noindent
-- The first map in (\ref{eqn:66a}) does a forward cyclic 
translation by $m$; that is, this map sends $\rho$ to 
$\widehat{\rho} = \{ \tau_m (V) : V \in \rho \}$,
with $\tau_m (k) = m+k \mbox{ mod} (n+1)$, 
$1 \leq k \leq n+1$.

\noindent
-- The second map in (\ref{eqn:66a}) merges together the 
numbers $m$ and $m+1$ in the block of $\widehat{\rho}$ 
which contains them.

We will use the notation
\[
F_n^{(m)} : \{ \rho \in NC(n+1) : \rho \ll 1_{n+1} \} \to NC(n)
\]
for the bijection obtained by composing the two maps from
(\ref{eqn:66a}).

It is useful to observe that for every $\rho \in NC(n+1)$
such that $\rho \ll 1_{n+1}$, one has
\begin{equation}   \label{eqn:66b}
\Moeb_n \bigl( F_n^{(m)} ( \rho) ,1_n \bigr) =
\Moeb_{n+1} (\rho,1_{n+1}).
\end{equation}
Indeed, if we follow the arrows 
$\rho \mapsto \widehat{\rho} \mapsto \pi = F_n^{(m)} ( \rho )$ 
in (\ref{eqn:66a}), one first has that 
$\Moeb_{n+1} ( \rho,1_{n+1}) 
= \Moeb_{n+1} ( \widehat{\rho},1_{n+1})$,
because the cyclic translation by $m$ gives an automorphism of
$NC(n+1)$ which preserves the values of the M\"obius function.
Then one has the equality 
$\Moeb_{n+1} ( \widehat{\rho},1_{n+1}) = \Moeb_n (\pi,1_n)$, 
which is seen by writing the explicit formulas of 
$\Moeb_{n+1} ( \widehat{\rho},1_{n+1})$ and of 
$\Moeb_n (\pi,1_n)$ (as in Remark \ref{def:24}), and by observing 
that the block structure of $K_{n+1} (\, \widehat{\rho} \, )$ only 
differs from the one of $K_n (\pi)$ by a singleton-block at $m$.
\end{remnotation}

\begin{lemma}   \label{lem:67}
Let $n \in \bN$ and $m \in \{ 1, \ldots , n \}$ be given.  Let us 
also fix a partition $\rho \in NC(n+1)$ such that $\rho \ll 1_{n+1}$,
and let us denote $\pi := F_n^{(m)} ( \rho ) \in NC(n)$.  We then have
\begin{equation}   \label{eqn:67a}
\gamma_{\pi}^{(m)} = \eta_{\rho} \circ \Gamma_{n+1}^m \circ 
\Delta_n^{(m)}
\end{equation}
(equality of multilinear functionals on $\cV^n$).
\end{lemma}

\begin{proof} 
The verification of Equation (\ref{eqn:67a}) is done by a mere unfolding
of the definitions of the functionals indicated on the two sides of
the equation.  As this unfolding only presents 
difficulties of notational 
nature, we believe it is more beneficial to the reader if we show it 
on a suitable concrete example which captures the relevant features of 
what is going on.  At the end of the proof we will elaborate on why 
the details of the verification are in fact general, rather than being 
specific to the example shown.

The concrete example we pick for illustration has 
\[
n=9, \, m=3, \, \rho = \bigl\{ \, \{ 1,4,9,10 \}, 
\, \{ 2,3 \}, \, \{ 5,6,8 \}, \, \{ 7 \} \, \bigr\} 
\in NC(10).
\]
Figure 4 below depicts the partition $\rho$, the partition
\[
\pi = F_9^{(3)} ( \rho ) = \bigl\{ \, \{ 1,7,8 \}, 
\, \{ 2,3,6 \}, \, \{ 4,5 \}, \, \{ 9 \} \, \bigr\} 
\in NC(9),
\]
and also the partition $\widehat{\rho} \in NC(10)$ which is used as 
an intermediate between $\rho$ and $\pi$ in Notation \ref{rem:66}.

$\ $

\begin{center}
  \setlength{\unitlength}{0.5cm}
(a)
  \begin{picture}(10,3)
  \thicklines
  \put(1,-1){\line(0,1){3}}
  \put(1,-1){\line(1,0){9}}
  \put(2,0){\line(0,1){2}}
  \put(2,0){\line(1,0){1}}
  \put(3,0){\line(0,1){2}}
  \put(4,-1){\line(0,1){3}}
  \put(5,0){\line(0,1){2}}
  \put(5,0){\line(1,0){3}}
  \put(6,0){\line(0,1){2}}
  \put(7,1){\line(0,1){1}}
  \put(8,0){\line(0,1){2}}
  \put(9,-1){\line(0,1){3}}
  \put(10,-1){\line(0,1){3}}
  \put(0.8,2.2){1}
  \put(1.8,2.2){2}
  \put(2.8,2.2){3}
  \put(3.8,2.2){4}
  \put(4.8,2.2){5}
  \put(5.8,2.2){6}
  \put(6.8,2.2){7}
  \put(7.8,2.2){8}
  \put(8.8,2.2){9}
  \put(9.8,2.2){10}
  \end{picture}
\hspace{0.8cm} (b)
  \begin{picture}(10,3)
  \thicklines
  \put(1,-1){\line(0,1){3}}
  \put(1,-1){\line(1,0){8}}
  \put(2,0){\line(0,1){2}}
  \put(2,0){\line(1,0){5}}
  \put(3,0){\line(0,1){2}}
  \put(4,0){\line(0,1){2}}
  \put(5,1){\line(0,1){1}}
  \put(5,1){\line(1,0){1}}
  \put(6,1){\line(0,1){1}}
  \put(7,0){\line(0,1){2}}
  \put(8,-1){\line(0,1){3}}
  \put(9,-1){\line(0,1){3}}
  \put(10,-1){\line(0,1){3}}
  \put(0.8,2.2){1}
  \put(1.8,2.2){2}
  \put(2.8,2.2){3}
  \put(3.8,2.2){4}
  \put(4.8,2.2){5}
  \put(5.8,2.2){6}
  \put(6.8,2.2){7}
  \put(7.8,2.2){8}
  \put(8.8,2.2){9}
  \put(9.8,2.2){10}
  \end{picture}               

\vspace{1cm}

(c)
  \begin{picture}(9,3)
  \thicklines
  \put(1,-1){\line(0,1){3}}
  \put(1,-1){\line(1,0){7}}
  \put(2,0){\line(0,1){2}}
  \put(2,0){\line(1,0){4}}
  \put(3,0){\line(0,1){2}}
  \put(4,1){\line(0,1){1}}
  \put(4,1){\line(1,0){1}}
  \put(5,1){\line(0,1){1}}
  \put(6,0){\line(0,1){2}}
  \put(7,-1){\line(0,1){3}}
  \put(8,-1){\line(0,1){3}}
  \put(9,-1){\line(0,1){3}}
  \put(0.8,2.2){1}
  \put(1.8,2.2){2}
  \put(2.8,2.2){3}
  \put(3.8,2.2){4}
  \put(4.8,2.2){5}
  \put(5.8,2.2){6}
  \put(6.8,2.2){7}
  \put(7.8,2.2){8}
  \put(8.8,2.2){9}
  \end{picture}

\vspace{1cm}

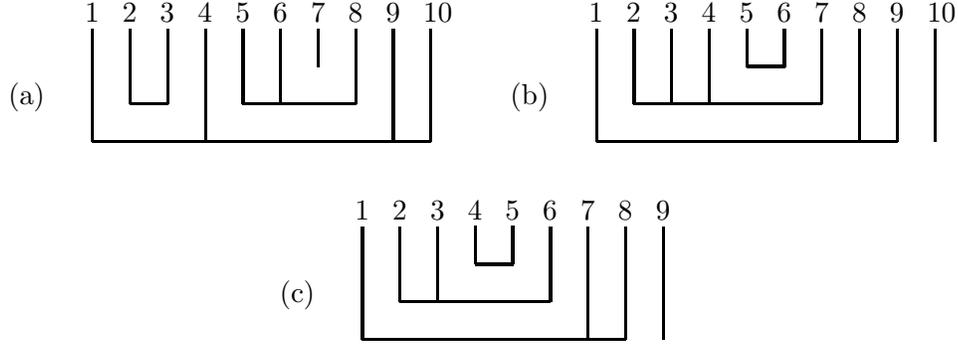
{\bf Figure 4.}  {\em (a) The partition
$\rho \ll 1_{10}$ used for illustration in this proof.}

{\em (b) The cyclic permutation 
$\widehat{\rho}$ of $\rho$ by $m=3$.  (c) The 
partition $\pi = F_9^{(3)} ( \rho ) \in NC(9)$.}
\end{center}

$\ $

Given a tuple $(x_1, \ldots , x_n) \in \cV^n$, the processing of
either side of Equation (\ref{eqn:67a}) will require the explicit
writing of $\Delta (x_m) \in \cV \otimes \cV$.  We will denote 
\[
\Delta (x_m) = \sum_{h=1}^s y_h \otimes z_h.
\]
In the concrete example used for illustration (with $n=9, m=3$), 
the block of $\pi$ which contains $m$ is $V_o = \{ 2,3,6 \}$, and 
the rank of $m$ in $V_o$ is $r=2$; hence 
the evaluation of the left-hand side of Equation (\ref{eqn:67a}) 
starts with
\begin{equation}   \label{eqn:67b}
\gamma_{\pi}^{(3)} (x_1, \ldots , x_9)
= \gamma_3^{(2)} (x_2, x_3, x_6) \cdot \varphi_3 (x_1, x_7, x_8)
\varphi_2 (x_4,x_5) \varphi_1 (x_9),
\end{equation}
where we then replace
\begin{align*}
\gamma_3^{(2)} (x_2, x_3, x_6) 
& = ( \beta_{4; \uChi} \circ \Gamma_4^2 \circ \Delta_3^{(2)} )
(x_2 \otimes x_3 \otimes x_6)
\mbox{ (cf. Notation \ref{def:62}(1))}                       \\
& = ( \beta_{4; \uChi} \circ \Gamma_4^2 )
(x_2 \otimes \Delta (x_3) \otimes x_6)                       \\
& = \sum_{h=1}^s \beta_{4; \uChi} \bigl( \, \Gamma_4^2 
(x_2 \otimes y_h \otimes z_h \otimes x_6) \bigr)             \\
& = \sum_{h=1}^s \beta_{4; \uChi} 
(z_h \otimes x_6 \otimes x_2 \otimes y_h)
= \sum_{h=1}^s \beta_{4; \uChi} (z_h , x_6 , x_2 , y_h)
\end{align*}
(with $\beta_{4; \uChi}$ interchangeably viewed as a linear 
functional on $\cV^{\otimes 4}$ or as a multilinear 
functional on $\cV^4$).  Hence the evaluation started in 
(\ref{eqn:67b}) comes to 
\begin{equation}   \label{eqn:67c}
\gamma_{\pi}^{(3)} (x_1, \ldots , x_9) =
\sum_{h=1}^s \beta_{4; \uChi} (z_h , x_6 , x_2 , y_h)
\cdot \varphi_3 (x_1, x_7, x_8)
\varphi_2 (x_4,x_5) \varphi_1 (x_9).
\end{equation}

Moving now to the corresponding evaluation on the right-hand 
side of Equation (\ref{eqn:67a}), we have
\[
( \eta_{\rho} \circ \Gamma_{10}^3 \circ \Delta_9^{(3)} )
(x_1, \ldots , x_9)
= ( \eta_{\rho} \circ \Gamma_{10}^3 )
(x_1 \otimes x_2 \otimes \Delta (x_3) \otimes x_4 \otimes 
\cdots \otimes x_9)
\]
\[
= \sum_{h=1}^s  \eta_{\rho} \bigl( \Gamma_{10}^3 
(x_1 \otimes x_2 \otimes y_h \otimes z_h \otimes x_4 \otimes 
\cdots \otimes x_9) \bigr)
\]
\[
= \sum_{h=1}^s  \eta_{\rho} (z_h \otimes x_4 \otimes
\cdots \otimes x_9 \otimes x_1 \otimes x_2 \otimes y_h )
= \sum_{h=1}^s  \eta_{\rho} 
(z_h , x_4 , \ldots , x_9, x_1, x_2 , y_h ).
\]
In the latter expression, we use the definition of 
$\eta_{\rho}$ from Notation \ref{def:64}; when doing 
so, it is instructive to re-draw the picture of $\rho$ 
from Figure 4, with the labels 
$1,2, \ldots, 10$ being replaced by labels 
``$z_h, x_4, \ldots , x_2, y_h$'':

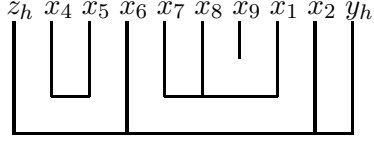
\begin{center}
  \setlength{\unitlength}{0.5cm}
  \begin{picture}(10,3)
  \thicklines
  \put(1,-1){\line(0,1){3}}
  \put(1,-1){\line(1,0){9}}
  \put(2,0){\line(0,1){2}}
  \put(2,0){\line(1,0){1}}
  \put(3,0){\line(0,1){2}}
  \put(4,-1){\line(0,1){3}}
  \put(5,0){\line(0,1){2}}
  \put(5,0){\line(1,0){3}}
  \put(6,0){\line(0,1){2}}
  \put(7,1){\line(0,1){1}}
  \put(8,0){\line(0,1){2}}
  \put(9,-1){\line(0,1){3}}
  \put(10,-1){\line(0,1){3}}
  \put(0.8,2.2){$z_h$}
  \put(1.8,2.2){$x_4$}
  \put(2.8,2.2){$x_5$}
  \put(3.8,2.2){$x_6$}
  \put(4.8,2.2){$x_7$}
  \put(5.8,2.2){$x_8$}
  \put(6.8,2.2){$x_9$}
  \put(7.8,2.2){$x_1$}
  \put(8.8,2.2){$x_2$}
  \put(9.8,2.2){$y_h$}
  \end{picture}
\end{center}

\vspace{0.75cm}

\noindent
The result of the evaluation for the right-hand side of Equation 
(\ref{eqn:67a}) then comes to
\begin{equation}   \label{eqn:67d}
( \eta_{\rho} \circ \Gamma_{10}^3 \circ \Delta_9^{(3)} )
(x_1, \ldots , x_9)
\end{equation}
\[
= \sum_{h=1}^s  \beta_{4; \uChi} (z_h , x_6 , x_2 , y_h )
\cdot \varphi_2 (x_4,x_5) \varphi_3 (x_7, x_8, x_1)
\varphi_1 (x_9).
\]
The expression obtained on the right-hand 
side of Equation (\ref{eqn:67d}) coincides with the one on 
the right-hand side of Equation (\ref{eqn:67c}), 
modulo the detail that the arguments $x_1, x_7, x_8$ were 
cyclically rotated in (\ref{eqn:67d}) (which doesn't change,
however, the value of $\varphi_3 (x_1, x_7, x_8)$, due to 
the hypothesis that $\uPhi$ is tracial).  This 
completes the verification of Equation (\ref{eqn:67a}) in 
the concrete example picked for illustration.

\vspace{6pt} 

We conclude our argument with a discussion about what 
is the general structure of the expressions encountered in 
(\ref{eqn:67c}) and (\ref{eqn:67d}), and what were the 
relations between the block structure of $\pi$ and of $\rho$ 
which produced the equality between these expressions.  We hope
that, upon examination, the reader will agree that this 
discussion is not specific to the special example considered 
above for illustration, but can be made whenever we consider 
partitions $\rho \in NC (n+1)$ and 
$\pi = F_n^{(m)} ( \rho ) \in NC(n)$ 
as described in the statement of the lemma.

-- We first note that the processing of either side of 
Equation (\ref{eqn:67a}) leads to an expression which is a 
sum of products.  The number $s$ of terms in the sum is picked
from an explicit writing of $\Delta (x_m)$, and every term has 
one factor which is a Boolean cumulant of $\uChi$, multiplied by
several factors which are moments of $\uPhi$.

-- The factors which are Boolean cumulants of $\uChi$ are found
by looking at the block $V_o = \{ i_1 < \cdots < i_p \} $
of $\pi$ which contains the number $m$ and, on the other hand,
at the unique outer block $W_o$ of $\rho$.  The sets $V_o$ and 
$W_o$ are related: one obtains $V_o$ out of $W_o$ via a cyclic 
rotation by $m$ followed by merging of $m$ with $m+1$, as 
explained in Notation \ref{rem:66}.  Both the processing using 
$V_o$ in (\ref{eqn:67c}) and the one using $W_o$ in 
(\ref{eqn:67d}) involve the same Boolean cumulants, which 
are 
\footnote{ It is fortunate that for these factors we don't 
need to do any kind of permutation in the arguments of the Boolean 
cumulants that appear.  Indeed, Boolean cumulants are not 
invariant under cyclic permutations of variables -- so even 
if we required $\uChi$ to be tracial, this feature would not be 
passed onto $\uBeta_{\uChi}$.}
of
the form
\begin{center}
$\beta_{p+1 ; \uChi} 
(z_h, x_{i_{r+1}}, x_{i_p}, x_{i_1}, 
\ldots , x_{i_{p-1}}, y_h)$, with $1 \leq h \leq s$.
\end{center}

-- The factors which are moments of $\uPhi$ are found
by looking at blocks $V \neq V_o$ of $\pi$ (for (\ref{eqn:67c}))
and at blocks $W \neq W_o$ of $\rho$ (for (\ref{eqn:67d})).
One has a natural bijective correspondence between such $V$'s 
and $W$'s, where every $V$ is obtained from the corresponding 
$W$ by a cyclic permutation followed by suitable relabelling.
(For instance, in the example of Figure 4: the block $W = \{ 2,3 \}$
of $\rho$ is cyclically permuted to $\{ 5,6 \}$ and then relabelled
to become the block $V = \{ 4,5 \}$ of $\pi$.)  For a $V$ and 
$W$ that correspond to each other: the definitions of the functionals
$\gamma_{\pi}^{(m)}$ and $\eta_{\rho}$ are made in such a way that
the choice of components selected out of a tuple 
$(x_1, \ldots , x_n) \in \cV^n$ is the same when we look at $V$ 
in (\ref{eqn:67c}) and when we look at $W$ in (\ref{eqn:67d}), 
modulo a possible cyclic permutation of the arguments.  This possible
cyclic permutation of the arguments is, however, taken care by the 
the hypothesis that $\uPhi$ is tracial.
\end{proof}

\begin{ad-hoc}
{\bf Conclusion of the proof of Theorem \ref{thm:17}.}
In view of Lemmas \ref{lem:63} and \ref{lem:65}, we are 
left to show that, for every $n \in \bN$, one has 
\begin{equation}   \label{eqn:68a}
\sum_{\pi \in NC(n)} \sum_{m=1}^n
\Moeb_n ( \pi , 1_n) \, \gamma_{\pi}^{(m)} =
\sum_{ \begin{array}{c}
       {\scriptstyle \rho \in NC(n+1),}  \\
       {\scriptstyle \rho \ll 1_{n+1} }  
       \end{array}  }
\sum_{m=1}^n
\Moeb_{n+1} ( \rho , 1_{n+1} ) \, \eta_{\rho} \circ
\Gamma_{n+1}^{(m)} \circ \Delta_n^{(m)}.
\end{equation}
And indeed, for any fixed $n \in \bN$ and $m \in \{ 1, \ldots , n \}$
we see that
\[
\sum_{\pi \in NC(n)} 
\Moeb_n ( \pi , 1_n) \, \gamma_{\pi}^{(m)}
= \sum_{ \begin{array}{c}
       {\scriptstyle \rho \in NC(n+1),}  \\
       {\scriptstyle \rho \ll 1_{n+1} }  
       \end{array}  }
\Moeb_n ( F_n^{(m)} ( \rho ) , 1_n ) 
\ \gamma_{ F_n^{(m)} ( \rho ) }^{(m)} 
\]
(via the change of summation variable 
``$\pi = F_n^{(m)} ( \rho )$'', based on Remark \ref{rem:66}),
which can be continued with
\[
= \sum_{ \begin{array}{c}
       {\scriptstyle \rho \in NC(n+1),}  \\
       {\scriptstyle \rho \ll 1_{n+1} }  
       \end{array}  }
\Moeb_{n+1} ( \rho , 1_{n+1} ) \ \eta_{\rho} \circ
\Gamma_{n+1}^{(m)} \circ \Delta_n^{(m)}               
\]
(by Lemma \ref{lem:67} and the equality 
of M\"obius functions observed in (\ref{eqn:66b})).
Summing over $m \in \{ 1, \ldots , n \}$ in the latter equalities
leads to (\ref{eqn:68a}).
\hfill $\square$
\end{ad-hoc}

$\ $

\section{Proof of Theorem \ref{thm:14}}

In this section we re-connect with the framework of
Section 1.2 of the Introduction, and we explain how 
Theorem \ref{thm:14} is a consequence of Theorem \ref{thm:17}.

$\ $

\begin{notation}   \label{def:71}
We fix throughout the section a positive integer $k$.  We put 
$\cV := \bC^k$, we fix a basis $e_1, \ldots , e_k$ of $\cV$,
and we consider the linear map $\Delta : \cV \to \cV \otimes \cV$
determined by the requirement that
\begin{equation}   \label{eqn:71a}
\Delta (e_i) = e_i \otimes e_i, \ \ 1 \leq i \leq k.
\end{equation}
We then consider the space $\fMcV$ of families of multilinear 
functionals defined as in Notation \ref{def:15}, and the 
transformation $\Deltastar : \fMcV \to \fMcV$ constructed by 
starting from $\Delta$ in the way shown in Notation \ref{def:16}.
It is immediate that, in the special case considered here, 
$\Deltastar$ acts as follows: given 
$\uPhi = ( \varphi_n )_{n=1}^{\infty} \in \fMcV$, one has 
$\Deltastar ( \uPhi ) = \uPsi = ( \psi_n )_{n=1}^{\infty}$, 
with $\psi_n$ determined via the requirement that for every 
$i_1, \ldots i_n \in \{ 1, \ldots ,k \}$ one has
\begin{equation}   \label{eqn:71b}
\psi_n (e_{i_1}, \ldots , e_{i_n} )
= \sum_{m=1}^n \varphi_{n+1} ( e_{i_m}, \ldots , e_{i_n},
e_{i_1}, \ldots , e_{i_m}).
\end{equation} 
\end{notation}

\begin{remnotation}   \label{rem:72}
In the present section, the notations for families of 
cumulants from Section 3 get to be used in two distinct ways.
On the one hand, we can consider families of functionals 
$\uPhi, \uChi , \uPhi ' \in \fMcV$; in connection to such
families we can define the four brands of cumulants discussed
in Section 3, which will be denoted here with appropriate
indices such as
\begin{equation}   \label{eqn:72a}
\left\{  \begin{array}{l}
\uKappa_{\uPhi} 
= ( \, \kappa_{n; \uPhi} \, )_{n=1}^{\infty}, 
\mbox{ for the free cumulants of $\uPhi$; }        \\
\uBeta_{\uChi} 
= ( \, \beta_{n; \uChi} \, )_{n=1}^{\infty},
\mbox{ for the Boolean cumulants of $\uChi$; }      \\
                                                    \\
\uKappa^{(c)}_{ ( \uPhi , \uChi )} 
= ( \, \kappa^{(c)}_{n; ( \uPhi , \uChi )} \, )_{n=1}^{\infty}, 
\mbox{ for the c-free cumulants of $\uPhi$ and $\uChi$; }        \\
\uKappa_{ ( \uPhi , \uPhi ' )} ' 
= ( \, \kappa_{n; ( \uPhi , \uPhi ')} ' \, )_{n=1}^{\infty}, 
\mbox{ for the infinitesimal free cumulants of 
       $\uPhi$ and $\uPhi '$. }      
\end{array}  \right.
\end{equation}
On the other hand, we can consider linear functionals 
$\mu , \nu , \mu '$ on the algebra of non-commutative 
polynomials $\bC \langle X_1, \ldots , X_k \rangle$,
with $\mu (1) = \nu (1) = 1$ and $\mu ' (1) = 0$.
In connection to such functionals: since
$\bC \langle X_1, \ldots , X_k \rangle$ is a unital 
algebra (where a linear functional naturally induces 
multilinear functionals on all the powers
$\bC \langle X_1, \ldots , X_k \rangle^n$) we can 
also define families of cumulants, with notations
such as
\begin{equation}   \label{eqn:72b}
\left\{  \begin{array}{l}
\uKappa_{\mu} 
= ( \, \kappa_{n; \mu} \, )_{n=1}^{\infty}, 
\mbox{ for the free cumulants of $\mu$; }                   \\
\uBeta_{\nu} 
= ( \, \beta_{n; \nu} \, )_{n=1}^{\infty},
\mbox{ for the Boolean cumulants of $\nu$; }                \\
                                                            \\
\uKappa^{(c)}_{ ( \mu , \nu )} 
= ( \, \kappa^{(c)}_{n; ( \mu , \nu )} \, )_{n=1}^{\infty}, 
\mbox{ for the c-free cumulants of $\mu$ and $\nu$; }        \\
\uKappa_{ ( \mu , \mu ' )} ' 
= ( \, \kappa_{n; ( \mu , \mu ')} ' \, )_{n=1}^{\infty}, 
\mbox{ for the infinitesimal free cumulants of 
       $\mu$ and $\mu '$. }      
\end{array}  \right.
\end{equation}

One has an obvious connection between the families of cumulants
listed in (\ref{eqn:72a}) and in (\ref{eqn:72b}), as recorded in 
the next lemma.
\end{remnotation}

\begin{lemma}  \label{lem:73}
Let $\mu , \nu , \mu ' : 
\bC \langle X_1, \ldots , X_k \rangle \to \bC$
be linear functionals with $\mu (1) = \nu (1) = 1$ and 
$\mu ' (1) = 0$. Consider the families of functionals 
$\uPhi , \uChi$ and $\uPhi '$ in $\fMcV$ defined via the 
requirement that for every $n \in \bN$ and 
$i_1, \ldots , i_n \in \{ 1, \ldots , k \}$ we have
\begin{equation}  \label{eqn:73a}
\left\{   \begin{array}{ll}
\varphi_n (e_{i_1}, \ldots , e_{i_n})   & =  
        \mu(X_{i_1} \cdots X_{i_n}),             \\
                                        &        \\
\chi_n (e_{i_1}, \ldots , e_{i_n})      & =  
    \nu(X_{i_1} \cdots X_{i_n}), \mbox{ and}     \\
                                        &        \\
\varphi'_n (e_{i_1}, \ldots , e_{i_n})  & = 
       \mu'(X_{i_1} \cdots X_{i_n}).
\end{array}  \right.
\end{equation}
Then, for every $n \in \bN$ and 
$i_1, \ldots , i_n \in \{ 1, \ldots , k \}$,
we have the following equalities of cumulants:

\vspace{6pt}

(1) $\kappa_{n ; \mu } 
(X_{i_1}, \ldots , X_{i_n}) =
\kappa_{n ; \uPhi}
(e_{i_1}, \ldots , e_{i_n})$.

\vspace{6pt}

(2) $\beta_{n ; \nu } 
(X_{i_1}, \ldots , X_{i_n}) =
\beta_{n ; \uChi}
(e_{i_1}, \ldots , e_{i_n})$.

\vspace{6pt}

(3) $\kappa_{n ; ( \mu , \nu )}^{(c)} 
(X_{i_1}, \ldots , X_{i_n}) =
\kappa_{n ; ( \uPhi , \uChi )}^{(c)}
(e_{i_1}, \ldots , e_{i_n})$.

\vspace{6pt}

(4) $\kappa_{n ; ( \mu , \mu ') } ' 
(X_{i_1}, \ldots , X_{i_n}) =
\kappa_{n ; ( \uPhi , \uPhi ' )} '
(e_{i_1}, \ldots , e_{i_n})$.
\end{lemma}

\begin{proof}  This is because, in each of the 
equations listed in the conclusion of the lemma,
the same moment-cumulant formulas are used on 
both sides of the equation.
\end{proof}

\begin{lemma}  \label{lem:74}
Consider the framework and notations 
of Lemma \ref{lem:73}.  Suppose we 
have $\mu ' = \Psi_k ( \nu )$, in the 
sense of Definition \ref{def:11}.
Then it follows that 
$\uPhi ' = \Deltastar ( \uBeta_{\uChi} )$, with 
$\Deltastar$ as reviewed in Notation \ref{def:71}.
\end{lemma}

\begin{proof}
Let $\Deltastar ( \uBeta_{\uChi} )
:= \uPsi = ( \psi_n )_{n=1}^{\infty}$.
For every $n \in \bN$ and 
$i_1, \ldots , i_n \in \{ 1, \ldots , k \}$
we write
\[
\psi_n (e_{i_1}, \ldots , e_{i_n} )
= \sum_{m=1}^n  \beta_{n+1 ; \uChi}
( e_{i_m}, \ldots , e_{i_n}, e_{i_1}, \ldots , e_{i_m} )
\mbox{ (by Equation (\ref{eqn:71b}))}
\]
\[ 
= \sum_{m=1}^n  \beta_{n+1 ; \nu}
( X_{i_m}, \ldots , X_{i_n}, X_{i_1}, \ldots , X_{i_m} )
\mbox{ (by Lemma \ref{lem:73}(2))}
\]
\[
= \mu ' (X_{i_1} \cdots X_{i_n} )
\mbox{ (by Equation (\ref{eqn:11a}) in 
        Definition \ref{def:11}) }
\]
\[
= \varphi_n ' (e_{i_1}, \ldots , e_{i_n})
\mbox{ (by the definition of $\varphi_n '$ in Eqn.
(\ref{eqn:73a})). }
\]
By multilinearity, it follows that 
$\psi_n = \varphi_n '$ for all $n \in \bN$. Hence
$\uPhi ' = \uPsi = \Deltastar ( \uBeta_{\uChi} )$.
\end{proof}

\begin{ad-hoc}
{\bf Proof of Theorem \ref{thm:14}.}
Recall that we have the following data:
$k$ is a positive integer, $\mu, \nu$ are in 
$\Dalg (k)$, and we let 
$\mu' := \Psi_k ( \nu ) \in \Dalg ' (k)$. 
We assume $\mu $ to be tracial.  We have to prove that 
the infinitesimal free cumulants of $( \mu , \mu' )$
are related to the $c$-free cumulants of $( \mu , \nu )$
by the formula 
\[
\kappa_{n; ( \mu , \mu' )} ' ( X_{i_1}, \ldots , X_{i_n} ) 
= \sum_{m=1}^n
\kappa^{(c)}_{n+1 ; ( \mu , \nu )}
( X_{i_m}, \ldots , X_{i_n}, X_{i_1}, \ldots , X_{i_m} ),
\] 
holding for every $n \in \bN$ and 
$i_1, \ldots , i_n \in \{ 1, \ldots , k \}$.

Consider the families of functionals 
$\uPhi, \uChi, \uPhi ' \in \fMcV$ defined as in 
Lemma \ref{lem:73}.  Note that the hypothesis of $\mu$ being
tracial immediately implies that $\uPhi$ is tracial as well. 
On the other hand, Lemma \ref{lem:74} gives us the fact 
that $\uPhi ' = \Deltastar ( \uBeta_{\uChi} )$.
Theorem \ref{thm:17} then applies to 
$\uPhi, \uChi, \uPhi '$, which leads to a formula relating
the infinitesimal free cumulants $\kappa_{n ; ( \uPhi , \uPhi ')}$ 
to the c-free cumulants $\kappa_{n ; ( \uPhi , \uChi )}^{(c)}$.
We are only left to write that, for every $n \in \bN$ and
$i_1, \ldots , i_n \in \{ 1, \ldots , k \}$:
\begin{align*}
\kappa_{n ; ( \mu ,  \mu ')} '
(X_{i_1}, \ldots , X_{i_n})  
& = \kappa_{n ; ( \uPhi , \uPhi ')} '
(e_{i_1}, \ldots , e_{i_n}) 
\mbox{ (by Lemma \ref{lem:74}(4)) }                       \\
& = \sum_{m=1}^n \kappa^{(c)}_{n+1 ; ( \uPhi , \uChi )} 
(e_{i_m}, \ldots , e_{i_n}, e_{i_1}, \ldots , e_{i_m}) 
\mbox{ (by Theorem \ref{thm:17})}                         \\
& = \sum_{m=1}^n \kappa^{(c)}_{n+1 ; ( \mu , \nu )}
(X_{i_m}, \ldots , X_{i_n}, X_{i_1}, \ldots , X_{i_m}),
\end{align*}
where at the latter equality we used
Lemma \ref{lem:74}(3).
\hfill  $\square$
\end{ad-hoc}

$\ $

\section{c-free and infinitesimally free
products and additive convolutions}

In this section we continue with the framework of algebraic 
distributions in $\Dalg (k)$ and $\Dalg ' (k)$, and with 
the various families of cumulants associated to such 
distributions, with 
notations as in (\ref{eqn:72b}) of the preceding section.
We will review the variations of the notions of free product 
and free additive convolution that are relevant for the 
present paper, and we will show how Theorems \ref{thm:12} 
and \ref{thm:13} follow from the formula about cumulants 
obtained in Theorem \ref{thm:14}.  
Since the latter formula only addresses cumulants 
for tuples of the special form $( X_{i_1}, \ldots , X_{i_n} )$ 
(rather than covering cumulants for 
tuples $(P_1, \ldots , P_n)$ with general 
$P_1, \ldots , P_n \in \bC \langle X_1, \ldots , X_k \rangle$)
it is useful to start by recording the following lemma.

$\ $

\begin{lemma}   \label{lem:81}
Let $k$ be a positive integer. 

\vspace{6pt}
 
(1) Let $\mu_1, \mu_2$ be in $\Dalg (k)$, and suppose that 
the free cumulant functionals of $\mu_1$ agree with those of 
$\mu_2$ on all tuples
$\{ ( X_{i_1}, \ldots , X_{i_n} ) : n \in \bN, 
1 \leq i_1, \ldots , i_n \leq k \}$.  Then $\mu_1 = \mu_2$

\vspace{6pt}
 
(2) Let $\mu, \nu_1, \nu_2$ be in $\Dalg (k)$, and suppose 
that the c-free cumulants of $( \mu , \nu_1 )$ agree with those 
of $( \mu , \nu_2 )$ on all tuples 
$\{ ( X_{i_1}, \ldots , X_{i_n} ) : n \in \bN, 
1 \leq i_1, \ldots , i_n \leq k \}$.  Then $\nu_1 = \nu_2$.

\vspace{6pt}
 
(3) Let $\mu \in \Dalg (k)$ and $\mu_1', \mu_2' \in \Dalg ' (k)$,
and suppose the infinitesimal free cumulants of 
$( \mu , \mu_1' )$ agree with those of $( \mu , \mu_2' )$ on 
all tuples $\{ ( X_{i_1}, \ldots , X_{i_n} ) : n \in \bN, 
1 \leq i_1, \ldots , i_n \leq k \}$.  Then $\mu_1 ' = \mu_2 '$.
\end{lemma}

\begin{proof}  Each of the three parts of the lemma follows by 
an immediate application of the suitable moment-cumulant 
formula.  For instance: in the case of statement (3), the 
formula (\ref{eqn:34a}) used in the definition of infinitesimal 
free cumulants gives us that both 
$\mu_1 ' ( X_{i_1} \cdots X_{i_n} )$ and
$\mu_2 ' ( X_{i_1} \cdots X_{i_n} )$ are equal to
\[
\sum_{\pi \in NC(n)} \sum_{V_o \in \pi} 
\kappa_{ | V_o | } '
\bigl( \, (X_{i_1}, \ldots , X_{i_n}) \mid V_o \bigr)
\cdot  \prod_{  \begin{array}{c}
{\scriptstyle V \in \pi}  \\
{\scriptstyle V \neq V_o}
\end{array} } \ \kappa_{ |V|} 
( \, (X_{i_1}, \ldots , X_{i_n}) \mid V ), \\
\]
where the $\kappa_{|V|}$s are free cumulant of $\mu$,
while $\kappa_{ | V_o | } '
\bigl( \, (X_{i_1}, \ldots , X_{i_n}) \mid V_o \bigr)$
is the common value of the infinitesimal free 
cumulants of $( \mu , \mu_1 ')$ and 
$( \mu , \mu_2 ')$ on the tuple 
$( X_{i_1}, \ldots , X_{i_n} ) \mid V_o$.  Thus $\mu_1 '$
and $\mu_2 '$ agree on all monomials 
$X_{i_1} \cdots X_{i_n}$ with  $n \in \bN$ and 
$i_1, \ldots , i_n \in \{ 1, \ldots , k \}$, 
and it follows that $\mu_1' = \mu_2'$, as required.
\end{proof}

$\ $

\subsection{Free product operations and the proof 
of Theorem \ref{thm:13}}

$\ $

\noindent
In this subsection we review the necessary extensions of the notion
of free product of distributions, and we show how Theorem \ref{thm:13}
follows from Theorem \ref{thm:14}.

For our purposes, it is most convenient to treat all versions of 
free products in terms of cumulants.  This starts with the standard 
free product operation, as described for instance in Lecture 6 of 
\cite{NiSp2006}.  Indeed, the free product $\mu_1 \star \mu_2$ of 
two distributions $\mu_1 \in \Dalg (k)$ and 
$\mu_2 \in \Dalg ( \ell )$ (for some $k, \ell \in \bN$) can be 
identified as the distribution 
$\widetilde{\mu} \in \Dalg ( k + \ell )$ which is uniquely 
determined by the requirement that its free cumulants fulfill 
the following condition: for every $n \in \bN$ and 
$i_1, \ldots , i_n \in \{ 1, \ldots , k + \ell \}$, one has
\begin{equation}   \label{eqn:8a}
\kappa_{n ; \widetilde{\mu}} 
(X_{i_1}, \ldots , X_{i_n}) = 
\left\{   \begin{array}{ll}
\kappa_{n; \mu_1} (X_{i_1}, \ldots , X_{i_n}), 
& \mbox{ if $i_1, \ldots , i_n \leq k$ }             \\
\kappa_{n; \mu_2} (X_{i_1 -k}, \ldots , X_{i_n -k}), 
& \mbox{ if $k < i_1, \ldots , i_n \leq k + \ell$ }  \\
0,  & \mbox{ otherwise.}
\end{array}  \right.
\end{equation}
For the proof of the fact that $\widetilde{\mu} := \mu_1 \star \mu_2$
satisfies the conditions (\ref{eqn:8a}), we refer to Lecture 11 of 
\cite{NiSp2006}.  The uniqueness of a distribution which satisfies 
(\ref{eqn:8a}) follows directly from Lemma \ref{lem:81}(1).
Let us also note here that when the formulas connecting moments to 
free cumulants are used in conjunction to Equation (\ref{eqn:8a}), 
it follows in particular that 
\begin{equation}   \label{eqn:8b}
\mu_1 \star \mu_2 \mid 
\bC \langle X_1, \ldots , X_k \rangle = \mu_1,
\end{equation}
and also that $\mu_1 \star \mu_2 \mid 
\bC \langle X_{k+1}, \ldots , X_{k + \ell} \rangle$ is obtained 
from $\mu_2$ via the relabelling $X_i \mapsto X_{k+i}$, 
$1 \leq i \leq \ell$.

The versions of free product operations used for c-free and for
infinitesimal free independence can be identified by using the 
same blueprint as above.  More precisely, we can proceed as 
follows.

\vspace{10pt}

\begin{prop-and-def}   \label{prop:82}
Let $k, \ell$ be positive integers.

\vspace{6pt}

(1) Let $\mu_1, \nu_1 \in \Dalg (k)$ and 
$\mu_2, \nu_2 \in \Dalg ( \ell )$.  We denote 
$\widetilde{\mu} := \mu_1 \star \mu_2 \in \Dalg (k + \ell )$.
There exists a distribution $\widetilde{\nu} \in \Dalg ( k + \ell )$, 
uniquely determined, such that the c-free cumulants of 
$( \widetilde{\mu}, \widetilde{\nu} )$ fulfill the 
following condition: for every $n \in \bN$ and 
$i_1, \ldots , i_n \in \{ 1, \ldots , k + \ell \}$, one has
\begin{equation}   \label{eqn:82a}
\kappa_{n ; ( \widetilde{\mu} , \widetilde{\nu} )}^{(c)} 
(X_{i_1}, \ldots , X_{i_n}) = 
\left\{   \begin{array}{ll}
\kappa_{n; ( \mu_1, \nu_1)}^{(c)} (X_{i_1}, \ldots , X_{i_n}),
& \mbox{ if $i_1, \ldots , i_n \leq k$ }             \\
\kappa_{n; ( \mu_2, \nu_2 )}^{(c)}
(X_{i_1 -k}, \ldots , X_{i_n -k}), 
& \mbox{ if $k < i_1, \ldots , i_n \leq k + \ell$ }  \\
0,  & \mbox{ otherwise.}
\end{array}  \right.
\end{equation}
The couple $( \widetilde{\mu} , \widetilde{\nu} )$ is called 
the {\em c-free product} of the couples $( \mu_1, \nu_1 )$ and 
$( \mu_2 , \nu_2 )$; we will use for it the notation
\begin{equation}  \label{eqn:82b}
( \widetilde{\mu} , \widetilde{\nu} )
= ( \mu_1 , \nu_1 ) \freeprodc ( \mu_2 , \nu_2 ).
\end{equation}

\vspace{6pt}

(2) Let $\mu_1 \in \Dalg (k)$, $\mu_1 ' \in \Dalg ' (k)$ and 
$\mu_2 \in \Dalg ( \ell )$, $\mu_2 ' \in \Dalg ' ( \ell )$.
We denote $\widetilde{\mu} := \mu_1 \star \mu_2 \in \Dalg (k + \ell )$.
There exists a distribution
$\widetilde{\mu} ' \in \Dalg ' ( k + \ell )$, 
uniquely determined, such that the infinitesimal free cumulants of 
$( \widetilde{\mu}, \widetilde{\mu} ' )$ fulfill the 
following condition: for every $n \in \bN$ and 
$i_1, \ldots , i_n \in \{ 1, \ldots , k + \ell \}$, one has
\begin{equation}   \label{eqn:82c}
\kappa_{n ; ( \widetilde{\mu} , \widetilde{\mu} ' )} '
(X_{i_1}, \ldots , X_{i_n}) = 
\left\{   \begin{array}{ll}
\kappa_{n; ( \mu_1, \mu_1 ')} ' (X_{i_1}, \ldots , X_{i_n}), 
& \mbox{ if $i_1, \ldots , i_n \leq k$ }             \\
\kappa_{n; ( \mu_2, \mu_2 ' )} '
(X_{i_1 -k}, \ldots , X_{i_n -k}), 
& \mbox{ if $k < i_1, \ldots , i_n \leq k + \ell$ }  \\
0,  & \mbox{ otherwise.}
\end{array}  \right.
\end{equation}
The couple $( \widetilde{\mu} , \widetilde{\mu} ')$ is called 
the {\em infinitesimal free product} of the couples 
$( \mu_1, \mu_1 ' )$ and $( \mu_2 , \mu_2 ' )$; we will use 
for it the notation
\begin{equation}   \label{eqn:82d}
( \widetilde{\mu} , \widetilde{\mu} ' ) 
= ( \mu_1 , \mu_1 ' ) \freeprodB ( \mu_2 , \mu_2 ' ).
\end{equation}
\end{prop-and-def}

\begin{proof}  (1) The existence of $\widetilde{\nu}$ is a 
special case of the general result about free product of 
C-ncps spaces proved in \cite{BoLeSp1996}.  Uniqueness follows 
from Lemma \ref{lem:81}(2).

(2) The existence of $\widetilde{\mu} '$ is a special case of 
the general construction of a free product of I-ncps spaces which 
is e.g. discussed in Section 2 of \cite{FeNi2010}.  Uniqueness 
follows from Lemma \ref{lem:81}(3).
\end{proof}

\begin{remark}   \label{rem:83}
(1) Analogously to what we had for $\mu_1 \star \mu_2$ in
Equation (\ref{eqn:8b}), the cumulant relations stated in 
Proposition \ref{prop:82}(1) imply in particular that 
$\widetilde{\nu} \mid \bC \langle X_1, \ldots , X_k \rangle = \nu_1$, 
and those from Proposition \ref{prop:82}(2) imply that 
$\widetilde{\mu} ' \mid \bC \langle X_1, \ldots , X_k \rangle = \mu_1 '$.  
Similar relations (modulo relabelling of $X_1, \ldots , X_{\ell}$ as
$X_{k+1}, \ldots , X_{k + \ell}$) hold in connection to the 
restrictions of $\widetilde{\nu}$ and $\widetilde{\mu} '$ to 
$\bC \langle X_{k+1}, \ldots , X_{k+ \ell} \rangle$.

\vspace{6pt}

(2) In Proposition \ref{prop:82} it was sufficient to prescribe
how the cumulant functionals
$\kappa_{n ; ( \widetilde{\mu} , \widetilde{\nu} )}^{(c)}$ and
$\kappa_{n ; ( \widetilde{\mu} , \widetilde{\mu} ' )} '$
act on tuples $(X_{i_1}, \ldots , X_{i_n} )$.  The action of
these functionals on general tuples $(P_1, \ldots , P_n) \in 
\Bigl( \, \bC \langle X_1, \ldots , X_k \rangle \, \Bigr)^n$
could also be explicitly described, if needed, by reducing 
(via multilinearity) to the case when $P_1, \ldots , P_n$ are 
monomials, and then by invoking the suitable formulas for 
cumulants with products as entries.  In order to get the latter
formulas for functionals
$\kappa_{n ; ( \widetilde{\mu} , \widetilde{\nu} )}^{(c)}$
one can combine Theorem 2.6 of \cite{Ca2000} with Proposition 
\ref{prop:54} of the present paper, while for functionals
$\kappa_{n ; ( \widetilde{\mu} , \widetilde{\mu} ' )} '$
one can use Propositions 3.15 and 4.3 from \cite{FeNi2010}.
\end{remark}

\begin{ad-hoc}
{\bf Proof of Theorem \ref{thm:13}.}
Recall that we have the following data: $k, \ell$ are positive integers
and we are given distributions $\mu_1, \nu_1  \in \Dalg (k)$,
$\mu_2, \nu_2 \in \Dalg ( \ell )$, such that $\mu_1,\mu_2$ are tracial.
We consider the free products
\begin{equation}   \label{eqn:84a}
( \mu_1 , \nu_1 ) \freeprodc ( \mu_2 , \nu_2 ) 
= ( \widetilde{\mu} , \widetilde{\nu} )
\in \Dalg ( k + \ell ) \times \Dalg ( k + \ell ), \mbox{ and}
\end{equation}
\begin{equation}   \label{eqn:84b}
( \mu_1 , \Psi_k ( \nu_1 ) ) \freeprodB 
( \mu_2 , \Psi_{\ell} ( \nu_2 ) ) 
= ( \widetilde{\mu} , \widetilde{\mu} ' )
\in \Dalg ( k + \ell ) \times \Dalg ' ( k + \ell ), 
\end{equation}
where the common $\widetilde{\mu}$ appearing in (\ref{eqn:84a}) and
(\ref{eqn:84b}) is the free product $\mu_1 \star \mu_2$.  
We have to prove that
$\Psi_{k + \ell} ( \widetilde{\nu} ) =  
\widetilde{\mu} '$.  In view of Lemma \ref{lem:81}(3), it will suffice
to verify that the equality
\begin{equation}   \label{eqn:84c}
\kappa_{n; 
( \widetilde{\mu} , \Psi_{k + \ell} ( \widetilde{\nu} ) )} '
(X_{i_1} , \ldots , X_{i_n}) =  
\kappa_{n; ( \widetilde{\mu} , \widetilde{\mu} ' )} '
(X_{i_1} , \ldots , X_{i_n})
\end{equation}
holds for all $n \in \bN$
and $1 \leq i_1, \ldots , i_n \leq k + \ell$.

For the verification of (\ref{eqn:84c}) we distinguish
three cases:

\vspace{6pt}

\noindent
{\em Case 1.} All of $i_1, \ldots , i_n$ are in 
$\{ 1, \ldots , k \}$.

\vspace{6pt}

\noindent
{\em Case 2.} All of $i_1, \ldots , i_n$ are in 
$\{ k+1, \ldots , k + \ell \}$.

\vspace{6pt}

\noindent
{\em Case 3.} We are not in Case 1 or Case 2.

\vspace{6pt}

In Case 1, the needed verification goes as follows:
\[
\kappa_{n; 
( \widetilde{\mu} , \Psi_{k + \ell} ( \widetilde{\nu} ) )} '
(X_{i_1} , \ldots , X_{i_n}) 
\]
\[
= \sum_{m=1}^n
\kappa^{(c)}_{n+1 ; ( \widetilde{\mu} , \widetilde{\nu} )}
(X_{i_m} , \ldots , X_{i_n}, X_{i_1}, \ldots , X_{i_m}) 
\ \ \mbox{ (by Theorem \ref{thm:14})}
\]
\[
= \sum_{m=1}^n
\kappa^{(c)}_{n+1 ; ( \mu_1 , \nu_1 )}
(X_{i_m} , \ldots , X_{i_n}, X_{i_1}, \ldots , X_{i_m}) 
\ \ \mbox{ (by Proposition \ref{prop:82}(1))}
\]
\[
= \kappa_{n ; ( \mu_1, \Psi_k ( \nu_1 ) )} '
(X_{i_1} , \ldots , X_{i_n}) 
\ \ \mbox{ (by Theorem \ref{thm:14})}
\]
\[
= \kappa_{n; ( \widetilde{\mu} , \widetilde{\mu} ' )} '
(X_{i_1} , \ldots , X_{i_n}),
\]
where at the latter equality we invoke Proposition 
\ref{prop:82}(2) in connection to how 
$( \widetilde{\mu} , \widetilde{\mu}  ' )$
is defined in (\ref{eqn:84b}).

\vspace{6pt}

The verifications of Cases 2 and 3 are analogous, where 
in Case 2 we must at some point shift the indices to
$i_1 -k , \ldots , i_n - k \in \{ 1, \ldots , \ell \}$, 
while in Case 3 both sides of the required equality 
(\ref{eqn:84c}) come out as equal to $0$.
\hfill  $\square$
\end{ad-hoc}

$\ $

We note that, when expressing the content of 
Theorem \ref{thm:13} directly in terms of the notions 
of c-free and infinitesimally free independence, one gets 
the following statement.

\begin{corollary}   \label{cor:85}
Let $k, \ell$ be positive integers, consider the algebra 
of polynomials $\cA =$
$ \bC \langle X_1, \ldots , X_{k+ \ell} \rangle$ and its 
subalgebras 
\[
\cA_1 = \bC \langle X_1, \ldots , X_k \rangle, 
\ \ \cA_2 = \bC \langle X_{k+1}, \ldots , X_{k + \ell} \rangle. 
\]
Let $\widetilde{\mu}, \widetilde{\nu}$ be in $\Dalg (k + \ell )$
and let $\widetilde{\mu} ' = \Psi_{k+ \ell} ( \widetilde{\nu} )
\in \Dalg ' (k + \ell )$.  If the subalgebras $\cA_1$ and $\cA_2$ 
are c-freely independent in the C-ncps 
$( \cA , \widetilde{\mu} , \widetilde{\nu} )$, then it follows 
that they are infinitesimally free independent in the I-ncps
$( \cA , \widetilde{\mu} , \widetilde{\mu} ' )$.
\end{corollary}

\begin{proof}
Let $\mu_1 , \nu_1 \in \Dalg (k)$ be the restrictions of 
$\widetilde{\mu}, \widetilde{\nu}$ to $\cA_1$, and let 
$\mu_2 , \nu_2 \in \Dalg ( \ell )$ be obtained from the 
restrictions of $\widetilde{\mu}, \widetilde{\nu}$ to $\cA_2$ 
by the natural shift of indices (for instance 
$\mu_2 (X_{i_1} \cdots X_{i_n}) :=
\widetilde{\mu} (X_{i_1 + k} \cdots X_{i_n + k})$ for every
$n \in \bN$ and $1 \leq i_1, \ldots , i_n \leq \ell$).
The hypothesis that $\cA_1$ and $\cA_2$ 
are c-freely independent in the C-ncps 
$( \cA , \widetilde{\mu} , \widetilde{\nu} )$ amounts 
to the fact that 
$( \widetilde{\mu} , \widetilde{\nu} ) =
( \mu_1 , \nu_1 ) \freeprodc ( \mu_2 , \nu_2 )$.  But then 
Theorem \ref{thm:13} tells us that 
$( \mu_1 , \Psi_k ( \nu_1 ) ) \freeprodB 
( \mu_2 , \Psi_{\ell} ( \nu_2 ) ) 
= ( \widetilde{\mu} , \widetilde{\mu} ' )$, and the latter 
equation converts into the conclusion about 
the infinitesimal free independence of $\cA_1$ and $\cA_2$ 
in the I-ncps $( \cA , \widetilde{\mu} , \widetilde{\mu} ' )$.
\end{proof}

$\ $

\subsection{ Operations of free additive convolution
and the proof of Theorem \ref{thm:12}}

$\ $

\noindent
For every $k \in \bN$ one has an operation $\boxplus$ of 
free additive convolution on $\Dalg (k)$, which reflects
the operation of adding two freely independent $k$-tuples 
of elements in a noncommutative probability space.  When 
going to free cumulants, the definition of $\boxplus$ 
takes the form of a ``linearization'' property: for 
$\mu_1, \mu_2 \in \Dalg (k)$, one has
\begin{equation}   \label{eqn:8c}
\kappa_{n; \mu_1 \boxplus \mu_2}
= \kappa_{n; \mu_1} + \kappa_{n; \mu_2}
\end{equation}
(equality of multilinear functionals on 
$\bC \langle X_1, \ldots , X_k \rangle^n$), 
holding for every $n \in \bN$ -- see e.g Lecture 16 in 
\cite{NiSp2006}.  In a context, such as the one of the present 
paper, where cumulants are the primary object of the discussion, 
the linearization property from Equation (\ref{eqn:8c}) can in 
fact be used in order to define the operation $\boxplus$.  The 
extensions of $\boxplus$ to c-free and to infinitesimal free 
probability can be approached in the same way, as indicated next.

$\ $

\begin{prop-and-def}   \label{prop:86}
Let $k$ be a positive integer.

\vspace{6pt}

(1) Let $\mu_1, \mu_2, \nu_1, \nu_2$ be in $\Dalg (k)$. 
We denote $\mu_1 \boxplus \mu_2 =: \mu \in \Dalg (k)$.
There exists a distribution $\nu \in \Dalg (k)$, 
uniquely determined, such that for every $n \in \bN$ 
and $i_1, \ldots , i_n \in \{ 1, \ldots , k \}$
one has
\begin{equation}   \label{eqn:86a}
\kappa_{n ; ( \mu, \nu )}^{(c)} 
(X_{i_1}, \ldots , X_{i_n} ) =
\kappa_{n ; ( \mu_1 , \nu_1 )}^{(c)} 
(X_{i_1}, \ldots , X_{i_n} ) 
+ \kappa_{n ; ( \mu_2 , \nu_2 )}^{(c)} 
(X_{i_1}, \ldots , X_{i_n} ).
\end{equation}
The couple $( \mu, \nu )$ is called 
the {\em c-free additive convolution} of the couples 
$( \mu_1, \nu_1 )$ and 
$( \mu_2 , \nu_2 )$; we will use for it the notation
\begin{equation}  \label{eqn:86b}
( \mu, \nu ) = ( \mu_1 , \nu_1 ) \boxplusc ( \mu_2 , \nu_2 ).
\end{equation}

\vspace{6pt}

(2) Let $\mu_1, \mu_2$ be in $\Dalg (k)$ and let 
$\mu_1 ' \mu_2 '$ be in $\Dalg ' (k)$.
We denote $\mu_1 \boxplus \mu_2 =: \mu \in \Dalg (k)$.
There exists a distribution $\mu ' \in \Dalg ' (k)$, 
uniquely determined, such that for every $n \in \bN$ 
and $i_1, \ldots , i_n \in \{ 1, \ldots , k \}$
one has
\begin{equation}   \label{eqn:86c}
\kappa_{n ; ( \mu, \mu ')} ' (X_{i_1}, \ldots , X_{i_n} ) 
= \kappa_{n ; ( \mu_1 , \mu_1 ' )} '
(X_{i_1}, \ldots , X_{i_n} ) 
+ \kappa_{n ; ( \mu_2 , \mu_2 ' )} '
(X_{i_1}, \ldots , X_{i_n} ).
\end{equation}
The couple $( \mu, \mu ')$ is called 
the {\em infinitesimally free additive convolution} of the 
couples $( \mu_1, \nu_1 )$ and 
$( \mu_2 , \nu_2 )$; we will use for it the notation
\begin{equation}  \label{eqn:86d}
( \mu, \mu ' ) = ( \mu_1 , \mu_1 ' ) 
\boxplusB ( \mu_2 , \mu_2 ' ).
\end{equation}
\end{prop-and-def}

\begin{proof}  The existence statement in both (1) and (2) of this 
proposition follow, by a standard argument, from the corresponding
parts of Proposition \ref{prop:82}.  (For instance in part (1),
one considers the free product 
$( \widetilde{\mu} , \widetilde{\nu} ) = 
( \mu_1 , \nu_1 ) \freeprodc ( \mu_2 , \nu_2 ) \in 
\Dalg (2k) \times \Dalg (2k)$,
and defines $\nu$ via the prescription that
$\nu (X_{i_1} \cdots X_{i_n}) =
\widetilde{\nu} \bigl( \, 
(X_{i_1}+X_{i_1 +k}) \cdots
(X_{i_n}+X_{i_n +k}) \, \bigr)$,
for all $n \in \bN$ and $1 \leq i_1 , \ldots , i_n \leq k$.)
The uniqueness statements concerning the required $\nu$ and $\mu '$ 
follow from parts (2) and respectively (3) of Lemma \ref{lem:81}.
\end{proof}

$\ $

\begin{ad-hoc}
{\bf Proof of Theorem \ref{thm:12}.}
Recall that we have the following data: $k$ is a positive integer
and $\mu_1, \nu_1, \mu_2, \nu_2$ are in $\Dalg (k)$, with 
$\mu_1, \mu_2$ tracial.  We denote 
\begin{equation}   \label{eqn:87a} 
( \mu_1 , \nu_1 ) \boxplusc 
( \mu_2 , \nu_2 ) = ( \mu , \nu ) 
\in \Dalg (k) \times \Dalg (k) \mbox{ and }
\end{equation}
\begin{equation}   \label{eqn:87b} 
( \mu_1 , \Psi_k ( \nu_1 ) ) \boxplusB 
( \mu_2 , \Psi_k ( \nu_2 ) ) 
= ( \mu , \mu ' ) 
\in \Dalg (k) \times \Dalg ' (k),
\end{equation}
where the common ``$\mu$'' appearing in (\ref{eqn:87a}) and in 
(\ref{eqn:87b}) is $\mu = \mu_1 \boxplus \mu_2$.
We have to prove that $\mu ' = \Psi_k ( \nu )$.  In view of 
Lemma \ref{lem:81}(3), it will suffice to verify that the 
infinitesimal free cumulants associated to $(\mu ,\mu')$ and to 
$(\mu , \Psi_{k}(\nu))$ agree on all tuples of the form 
$(X_{i_1}, \ldots , X_{i_n})$.  In order to do this verification, 
we can proceed as follows: we start from the equality
\begin{equation}   \label{eqn:87c} 
\kappa_{n; (\mu ,\mu')} ' (X_{i_1},\ldots,X_{i_n}) 
= \kappa_{n; (\mu_1 , \Psi_k ( \nu_1 )}' (X_{i_1},\ldots,X_{i_n})              
+ \kappa_{n; (\mu_2 , \Psi_k ( \nu_2 )}' (X_{i_1},\ldots,X_{i_n})
\end{equation}
(which holds by Equation (\ref{eqn:87b}) and the additivity 
property of $\boxplusB$), and we expand both cumulants on the 
right-hand side of (\ref{eqn:87c}) as sums ``$\sum_{m=1}^n$'' 
in the way indicated by Theorem \ref{thm:14}.  The sum of the 
two resulting ``$\sum_{m=1}^n$'' can be further processed by 
using the fact that  the functionals 
$\kappa_{n+1 ; (\mu , \nu)}^{(c)}$ and
$\kappa_{n+1 ; (\mu_1 , \nu_1)}^{(c)} +
\kappa_{n+1 ; (\mu_2 , \nu_2)}^{(c)}$ agree on tuples of the 
form $(X_{i_m}, \ldots , X_{i_n}, X_{i_1}, \ldots , X_{i_m})$,
with $1 \leq m \leq n$.  In this we arrive to the equality 
\[
\kappa_{n; (\mu ,\mu')} ' (X_{i_1},\ldots,X_{i_n}) 
= \sum_{m=1}^{n} \kappa_{n+1 ; (\mu , \nu)}^{(c)} 
(X_{i_m},\ldots,X_{i_n},X_{i_1},\ldots,X_{i_m});
\]
the latter sum is indeed equal to
$\kappa_{n ; (\mu , \Psi_{k}(\nu)) } ' (X_{i_1},\ldots,X_{i_n})$,
by another application of Theorem \ref{thm:14}.
\hfill  $\square$
\end{ad-hoc}

$\ $

\end{document}